\documentclass[12pt,a4paper]{article}%
\usepackage{amssymb}
\usepackage{amsmath}
\usepackage{amsfonts}
\usepackage[dvips]{graphicx}
\usepackage{mathrsfs}
\usepackage[pdftex]{hyperref}
\usepackage{pdfpages}%
\usepackage[all]{xy}
\providecommand{\U}[1]{\protect\rule{.1in}{.1in}}
\hypersetup{
	colorlinks=true,
	citecolor = blue
}
\newtheorem{theorem}{Theorem}

\newtheorem{corollary}[theorem]{Corollary}

\newtheorem{definition}[theorem]{Definition}

\newtheorem{lemma}[theorem]{Lemma}

\newtheorem{proposition}[theorem]{Proposition}
\newtheorem{remark}[theorem]{Remark}

\newenvironment{proof}[1][Proof]{\noindent\textbf{#1.} }{\ \rule{0.5em}{0.5em}}

\begin{document}
	
	\author{V\'{\i}ctor A. Vicente-Ben\'{\i}{}tez\\{\small Instituto de Matemáticas de la U.N.A.M. Campus Juriquilla}\\{\small Boulevard Juriquilla 3001, Juriquilla, Querétaro C.P. 076230 M\'{e}xico }\\ {\small  va.vicentebenitez@im.unam.mx } }
	\title{Transmutation operators for Schr\"odinger equations with distributional potentials and the associated  impedance equation}
	\date{}
	\maketitle
	
	\begin{abstract}
		We present the construction of an integral transmutation operator for the Schr\"odinger equation 
		\[
		-y'' + q(x)y = \lambda y, \quad x \in J, \ \lambda \in \mathbb{C},
		\]
		in the case where $q$ is the distributional derivative of an $L^2$ function on a bounded interval $J \subset \mathbb{R}$. Such a transmutation operator transforms solutions of $ v'' + \lambda v = 0 $ into solutions of the Schr\"odinger equation. The construction of the integral transmutation operator relies on a new regularization of the distributional Schr\"odinger equation based on the Polya factorization in terms of a solution $f$ that does not vanish on the closure of $J$. The existence of such a function $f$ is established, together with a constructive method for its computation.
		
		As a consequence of the Polya factorization, we obtain an integro-differential transmutation operator for the associated Sturm--Liouville operator in impedance form related to $f$, along with smoothness conditions for the transmutation kernel. Furthermore, we introduce the Darboux transform for both the Schr\"odinger and impedance operators and describe their relationships with the corresponding transmutation operators. Finally, we develop several series representations for the solutions, including the spectral parameter power series and the Neumann series of spherical Bessel functions.
	\end{abstract}
	
	\newline 
	
	
	\section{Introduction}
	
	The theory of transmutation operators, more than fifty years after its creation based on the work of Jacques Delsarte \cite{delsarte}, has demonstrated its effectiveness not only as a theoretical tool for the qualitative analysis of solutions to Sturm--Liouville-type equations, but also for their practical construction and for solving both direct and inverse boundary value problems. Roughly speaking, a transmutation operator is a linear homeomorphism $\mathbf{T}$ on a given topological vector space that relates two linear operators $\mathbf{A}$ and~$\mathbf{B}$ defined on a subspace through a similarity relation of the form $\mathbf{A}\mathbf{T} = \mathbf{T}\mathbf{B}$. The idea is to transmute the simpler operator $\mathbf{B}$ into the more complicated one $\mathbf{A}$. In particular, the operator $\mathbf{T}$ transforms solutions of $\mathbf{B}v = \lambda v$ into solutions of $\mathbf{A}u = \lambda u$. Of particular interest is the case when $\mathbf{B} = -\frac{d^2}{dx^2}$ and $\mathbf{A}$ is the one-dimensional Schr\"odinger operator $-\frac{d^2}{dx^2} + q(x)$, where $q$ is a {\it regular function}, that is, an $L^p$-function in a bounded interval. A first contribution to the construction of such transmutation operators was made by A. Ya. Povzner \cite{povzner}, who discovered that certain solutions of the Schr\"odinger equation 
	\begin{equation}\label{eq:schrodingerintro}
		-y''+q(x)y=\lambda y,\quad 0<x<\ell,
	\end{equation}
	admit the integral representation
	\begin{equation}\label{eq:opratorintro}
		y(\lambda,x)= e^{i\sqrt{\lambda}x}+\int_{-x}^xK(x,t)e^{i\sqrt{\lambda}t}dt.
	\end{equation}
	
	Later studies focused on analyzing the properties of the transmutation integral kernel $K$ \cite{bondarenko,camposlbases,camposstandard,carroll,hryniv1,minedeltas,marchenko,NSBF1}, one of the most significant results being its application to the theory of inverse spectral problems via the Gelfand-Levitan equation \cite{gelfand}.
	
	In recent years, the study of integral representations of the form \eqref{eq:opratorintro} has gained particular importance in the development of series-type analytical representations for the solutions of equation \eqref{eq:schrodingerintro}. One of the most significant advances is the explicit characterization of certain families of transmutation operators acting on the set of monomials $\{x^k\}_{k=0}^{\infty}$. The construction of such transmutation operators is based on the so-called {\it Polya factorization} of the Schr\"odinger operator. This can be described as follows: given a continuous solution $f$ of \eqref{eq:schrodingerintro}, whose regularity depends on the potential $q$, and which has no zeros in the interval of interest, the Schr\"odinger operator can be expressed through the following factorization:
	\begin{equation}\label{eq:polyaintro}
		-\frac{1}{f(x)}\frac{d}{dx}\left(f^2(x)\frac{d}{dx}\left(\frac{y(x)}{f(x)}\right)\right).
	\end{equation}
	The existence of the nonvanishing solution was established for the case where $q$ is continuous in \cite{spps} and later extended to $L^2$ potentials perturbed by a sum of delta point interactions in \cite{minedeltas}. This factorization makes it possible to obtain a complete system of solutions in terms of recursive integrals of $f$, denoted by $\{\varphi_f^{(k)}\}_{k=0}^{\infty}$ and referred to as the {\it formal powers associated with} $f$ \cite{spps}. In fact, the solutions of equation \eqref{eq:schrodingerintro} can be expressed as power series in the spectral parameter $\rho=\sqrt{\lambda}$, whose coefficients in $x$ are precisely the formal powers. This result is known as the {\it SPPS method} and has proven to be an effective tool for solving direct spectral problems, such as the computation of eigenvalues and eigenfunctions (see \cite{spps,sppscampos}). Moreover, the formal powers form a complete system in $L^p(-\ell,\ell)$ and, under certain regularity conditions on $f$, also in $W^{k,2}(-\ell,\ell)$ \cite{KravTremblay1,mineimpedance3}. Furthermore, when the potential $q$ is continuous, there exists a transmutation kernel $K_f(x,t)$ satisfying the Goursat conditions $K_f(x,x)=\frac{f'(0)}{2}+\frac{1}{2}\int_0^xq(s)ds$, $K_f(x,-x)=\frac{f'(0)}{2}$ (see \cite[Ch. 1]{marchenko}). The corresponding transmutation operator, denoted by $\mathbf{T}_f$, satisfies the mapping property \cite{camposlbases} $\mathbf{T}_f[x^k]=\varphi_f^{(k)}(x)$ for all $k\in \mathbb{N}_0$. Thus, the action of the operator $\mathbf{T}_f$ on a dense set of functions is explicitly known. This result was extended to $L^p$-type potentials in \cite{camposstandard}. In \cite{NSBF1} it was shown that the transmutation kernel admits a Fourier-Legendre series representation of the form $K_f(x,t)=\sum_{m=0}^{\infty}\frac{a_m(x)}{x}P_m\left(\frac{t}{x}\right)$, where the coefficients $\{a_m(x)\}_{m=0}^{\infty}$ can be computed explicitly by a recursive integration procedure. Moreover, it is known that $q(x)=\frac{a_0''(x)}{a_0(x)+1}$. This series representation of the kernel leads to the so-called {\it Neumann series of Bessel functions (NSBF) representation} for the solutions of \eqref{eq:schrodingerintro}, that is, the solutions $u_0(\rho,x)=\mathbf{T}_f[\cos(\rho x)]$ and $u_1(\rho,x)=\mathbf{T}_f[\sin(\rho x)]$ admit the following series representations:
	\begin{equation}\label{eq:nsbfintro}
		u_j(\rho,x)=\frac{1+(-1)^j}{2}\cos(\rho x)+\frac{1-(-1)^j}{2}\sin(\rho x)+\sum_{k=0}^{\infty}(-1)^ka_{2k+j}(x)j_{2k+j}(\rho x), \quad j=0,1,
	\end{equation}
	where $j_{\nu}(\zeta)=\sqrt{\frac{\pi}{2\zeta}}J_{\nu+\frac{1}{2}}(\zeta)$ stands for the spherical Bessel functions (see \cite{NSBF1}). These types of representations are particularly useful for describing spectral data, such as the characteristic function of an eigenvalue problem, and have proven to be a powerful tool for the numerical computation of eigenvalues and eigenfunctions. Moreover, in the analysis of inverse spectral problems (namely, the recovery of the potential and boundary conditions from certain spectral data such as pairs of eigenvalues and normalization constants or the Weyl function), the NSBF series have demonstrated great versatility for the practical solution of such problems. In \cite{vkinverse1} it was shown that the existence of the NSBF series \eqref{eq:nsbfintro} allows one to reduce several inverse problems to the solution of a system of linear algebraic equations whose unknowns are the coefficients $\{a_m(x)\}_{m=0}^{\infty}$. Given the relationship of the potential $q$ and the coefficient $a_0$, it is possible to recover $q$ from the first component of the solution vector. This approach is highly efficient for numerical computations and has been successfully applied to various types of inverse spectral problems in recent years \cite{vkinverse1,vkinverse2,vkstinverse2}.
	
	The aim of this paper is to extend the construction of the transmutation operator $\mathbf{T}_f$ and the corresponding series representations of the solutions to the case when $q$ is a distribution belonging to the class $W^{-1,2}$. Integral representations for the solutions $u_j(\rho,x)$ are already known \cite{bondarenko,shkalikov1}, as well as a transmutation relation in the sense of Delsarte is known for Schr\"odinger operators with Dirichlet and Neumann conditions \cite{alveberio,hryniv1}. The standard approach to regularizing operators with distributional coefficients, established in \cite{shkalikov1}, consists of rewriting the Schr\"odinger operator in terms of the quasi-derivative defined by an antiderivative $\sigma\in L^2$ of the potential $q$. In this work, we adopt a different approach, introducing a new type of regularization based on Polya’s factorization. To this end, we show that for every distributional potential there always exists a nonvanishing solution $f$, and that Polya's factorization is valid for functions $y\in W^{1,2}$ for which the Schrödinger operator defines a regular distribution on $L^2$. Moreover, we demonstrate that the domain of regularization of the Schr\"odinger operator consists of those functions $y\in W^{1,2}$ satisfying $y'-\frac{f'}{f}y\in W^{1,2}$. Our result encompasses potentials given by complex Radon measures and extends the results of \cite{minedeltas} for $L^2$ potentials perturbed by finitely many delta point interactions.

	Using Polya's factorization, we reduce equation \eqref{eq:schrodingerintro} to a Dirac-type system of equations. Based on the results in \cite{haratumyant}, we establish the existence of the integral representation \eqref{eq:opratorintro}. The main properties of the operator $\mathbf{T}_f$ are analyzed, including its continuity from $L^2(-\ell,\ell)$ onto $ L^2(0,\ell)$ and from $C[-\ell,\ell]$ onto $C[0,\ell]$. One of our principal results is the transmutation mapping property $\mathbf{T}_f[x^k]=\varphi_f^{(k)}$, together with the SPPS and  NSBF representations for the solutions of \eqref{eq:schrodingerintro}. Furthermore, we show that the formal powers $\{\varphi_f^{(k)}\}_{k=0}^{\infty}$ form an $L$-basis for the distributional Schr\"odinger operator \cite{camposstandard,fage}. From the transmutation mapping property, we demonstrated that $\mathbf{T}_f$ transmutes the operator $-\frac{d^2}{dx^2}$ into the distributional Schr\"odinger operator in the sense of Delsarte. As a consequence of Polya's factorization, we generalize the concept of the {\it Darboux transformed operator}, that is, the corresponding distributional equation associated with $\frac{1}{f}$, and we establish a relation between the operator $\mathbf{T}_f$ and its associated Darboux transform $\mathbf{T}_{\frac{1}{f}}$. The Darboux transform and its related factorization are widely used in quantum mechanics to construct pairs of supersymmetric potentials \cite{cooper} and constitute a powerful tool for the construction of exactly solutions of Sturm-Liouville problems \cite{guliyev} and transmutation operators for partial differential equations \cite{minedarboux}.
	
	The second part of this work is devoted to the construction of a transmutation operator for the Sturm-Liouville equation in impedance form $\mathbf{L}_f=-\frac{1}{f^2}\frac{d}{dx}f^2\frac{d}{dx}$. Polya's factorization establishes the equivalence between the Schr\"odinger and impedance operators via the Liouville transform $y\mapsto \frac{y}{f}$. Integral representations for the solutions of the impedance equation have been studied in \cite{alveberio,carroll, mineimpedance2}, where their applications to inverse spectral problems are explored. For the case when $f$ is of class $C^1$, it was shown in \cite{mineimpedance1} that there exists a transmutation operator that maps the second derivative into the impedance operator, in the form of the integro-differential operator
	\[
	\widehat{\mathbf{T}}_fu(x)=u(x)-\int_{-x}^{x}\widehat{K}_f(x,t)u'(t)dt.
	\]
	This result was later extended to the case $f\in W^{1,\infty}$ in \cite{mineimpedance3}. We extend this result to the case when $f\in W^{1,2}$. With the aid of the Liouville transform and the operator $\mathbf{T}_f$, we show that $\widehat{\mathbf{T}}_f$ is a transmutation operator in the sense of Delsarte, and that the kernel $\widehat{K}_f$ is continuous and belongs to the class $W^{1,2}$ in the triangular domain $0\leq x\leq \ell$, $|t|\leq x$. Furthermore, we prove that $\widehat{K}_f$ satisfies a hyperbolic equation with the Goursat conditions $\widehat{K}_f(x,x)=1-\frac{1}{f(x)}$, $\widehat{K}_f(x,-x)=0$. Additionally, we show that if $\{f_n\}$ is a sequence of impedance functions such that $\frac{f_n'}{f}\rightarrow \frac{f'}{f}$ in $L^2$, then the corresponding transmutation kernels $\{\widehat{K}_{f_n}\}$ converge weakly to $\widehat{K}_f$ in $W^{1,2}$, and $\widehat{\mathbf{T}}_{f_n}$ converges to $\widehat{\mathbf{T}}_f$ in the strong operator topology of $\mathcal{B}(W^{1,2}(-\ell,\ell),W^{1,2}(0,\ell))$. 
	
	The paper is structured as follows. Section 2 presents the correct definition of the distributional Schr\"odinger operator and equation, the concept of domain of regularization, and the standard regularization in terms of the antiderivative $\sigma$ of $q$. Section 3 introduces the new regularization based on the Polya factorization given by the nonvanishing solution $f$, together with a new characterization of the domain of regularization and the quasiderivative in terms of $f$. Section 4 presents the spectral parameter power series representation for the solutions in terms of the formal powers associated with $f$. In Section 5, we introduce the Darboux transform and analyze its main properties. Section 6 is devoted to the construction of the transmutation operator $\mathbf{T}_f$ and the study of its main analytical properties. Section 7 focuses on proving the existence of the nonvanishing solution $f$ and provides an algorithm for its construction. Section 8 presents approximation results for the kernel $K_f$ and the transmutation operator $\mathbf{T}_f$. In Section 9, we develop the Neumann series of Bessel functions representation for the solutions of the Schr\"odinger equation. Finally, Section 10 is devoted to the construction of a transmutation operator for the Sturm-Liouville operator in impedance form and to the analysis of the corresponding transmutation kernel.
	
	\section{Background on the distributional Schr\"odinger equation}
	 Let $X$ be a topological linear space and $X'$ its topological dual. The action of a functional $f\in X'$ over $x\in X$ will be denoted by $(f|x)_X$.
	Let $J\subset \mathbb{R}$ be a bounded open interval,  $C_0^{\infty}(J)$ the space of $C^{\infty}$ functions with compact support in $J$, and  $(C_0^{\infty}(J))'$ its topological dual, whose elements are called distributions.  A distribution $f\in (C_0^{\infty}(J))'$ is said to be {\it regular} if there exists $\tilde{f}\in L^1_{loc}(J)$ such that $(f|\phi)_{C_0^{\infty}(J)}=\int_J\tilde{f}\phi$ for all $\phi\in C_0^{\infty}(J)$, and $f$ is called $L^p$-regular if $\tilde{f}\in L^p(J)$.  Let $W^{k,p}(J)$, $k\in \mathbb{N}$, denote the Sobolev space of functions in $L^p(J)$ whose first $k$-distributional derivatives are $L^p$-regular. It is known that $u\in W^{k,p}(J)$ iff $u\in C^{k-1}(\overline{J})$, $u^{(k-1)}\in AC(\overline{J})$ and $u^{(k)}\in L^p(J)$. We recall the continuous embedding $W^{1,p}(J)\hookrightarrow C(\overline{J})$. According to \cite[Cor. 8.10]{brezis} $uv\in W^{1,p}(J)$ whenever $u,v\in W^{1,p}(J)$. The space $W_0^{1,p}(J)$ is the closure of $C_0^{\infty}(J)$ in $W^{1,p}(J)$ and consists of functions $u\in W^{1,p}(J)$ such that $u|_{\partial J}=0$.   The space $W^{-1,p'}(J)$ is the topological dual of $W_0^{1,p}(J)$. Since $J$ is bounded, \cite[Prop. 8.14]{brezis} implies that every $f\in W^{-1,p'}(J)$ can be represented as the distributional derivative of function $F \in L^{p'}(J)$, that is, 
	\[
	(f|\phi)_{W_0^{1,p}(J)}=-\int_JF\phi' \qquad \forall \phi \in W_0^{1,p}(J).
	\]
	Similar notation is used for Sobolev spaces $W^{k,p}(\Omega)$ on open subsets of $\mathbb{R}^N$. Given Banach spaces $X$ and $Y$, we denote by $\mathcal{B}(X,Y)$ the space of bounded linear operators from $X$ to $Y$. When $X=Y$, we denote this space by $\mathcal{B}(X)$.  We recall that a subset $E\subset X$ is said to be complete in $X$ if $\operatorname{Span}E$ is dense in $X$. The characteristic function of a set is denoted by $\chi_A$. Throughout the paper, we use the notation $\mathbb{N}_0=\mathbb{N}\cup\{0\}$.  The weak convergence of a sequence $\{x_n\}$ in a Banach space is denoted by $x_n\rightharpoonup x$, while convergence in the norm (strong convergence) is denoted by $x_n\rightarrow x$.
	\newline
	
	Let $q\in W^{-1,2}(J)$. We consider the Schr\"odinger equation with distributional potential $q$:
	\begin{equation}\label{eq:Schrodinger1}
		-y''+q(x)y=\lambda y,\quad x\in J, \; \lambda\in \mathbb{C}.
	\end{equation}
	We look for solutions $y\in W^{1,2}(J)$. The left-hand side of \eqref{eq:Schrodinger1}, denoted by $\mathbf{S}_qy$, defines a distribution in $W_0^{1,2}(J)$ as follows:
	\begin{equation}\label{eq:actionopS}
		(\mathbf{S}_qy|\phi)_{W^{1,2}_0(J)}:= -(y''|\phi)_{W^{1,2}_0(J)}+(yq|\phi)_{W^{1,2}_0(J)} =\int_Jy'\phi'+(q|y\phi)_{W^{1,2}_0(J)}. 
	\end{equation}
	Since $y\phi\in W_0^{1,2}(J)$, the pairing $(q|y\phi)_{W^{1,2}(J)}$ is well defined. Hence, we obtain the operator $\mathbf{S}_q: W ^{1,2}(J)\rightarrow W^{-1,2}(J)$ whose action is given by \eqref{eq:actionopS}. Of course, a solution $y\in W^{1,2}(J)$ of \eqref{eq:Schrodinger1} satisfies the condition that $\mathbf{S}_qy=\lambda y$ as distributions, so $\mathbf{S}_qy$ is an $L^2$-regular distribution.
	\begin{definition}
		The {\bf domain of} $L^2$-{\bf regularization} of $\mathbf{S}_q$ is the class
		\[
		\mathscr{D}_2(\mathbf{S}_q):=\{y\in W^{1,2}(J)\,|\, \mathbf{S}_qy \text{ is an } L^2\text{-regular distribution}\}.
		\]
	\end{definition}
	Abusing notation, when $y\in \mathscr{D}_2(\mathbf{S}_q)$ we write $\mathbf{S}_qy\in L^2(J)$. In this case, the solutions of \eqref{eq:Schrodinger1} belong to $\mathscr{D}_2(\mathbf{S}_q)$.

	A characterization of the action of $\mathbf{S}_qy$ can be obtained in terms of an antiderivative of $q$. Let  $\sigma \in L^2(J)$ satisfying $\sigma'=q$. Hence, the action of the distribution $\mathbf{S}_qy$ over $\phi\in W^{1,2}_0(J)$ is given by
	\begin{align*}
		(\mathbf{S}_qy|\phi)_{W^{1,2}_0(J)} & =  \int_J \{ y'\phi'-\sigma (y\phi)'\}= \int_J \{y'\phi'-\sigma y'\phi -\sigma y\phi'\} \\
		&= \int_J \{(y'-\sigma y)\phi'-\sigma (y'-\sigma y)\phi-\sigma^2y\phi\}
	\end{align*}
	
	Following the notation introduced in \cite{shkalikov1}, we denote the $\sigma$-quasiderivative of $y\in W^{1,2}(J)$ by
	
	\begin{equation}\label{eq:quasiderivative}
		y_{\sigma}^{[1]}:= y'-\sigma y.
	\end{equation}

	From the previous computations, the following proposition is immediate.
	\begin{proposition}
		Given $y\in W^{1,2}(J)$, the distribution $\mathbf{S}_qy$ is $L^2$-regular iff $y_{\sigma}^{[1]}\in W^{1,1}(J)$ and $(y_{\sigma}^{[1]})'+\sigma^2y\in L^2(J)$.
	\end{proposition}

	Another reason for considering functions for which the operator is $L^2$-regular is to obtain an operator in the Hilbert space $L^2(J)$. Abusing notation, for $y\in \mathscr{D}_2(\mathbf{S}_q)$ we write $\mathbf{S}_qy=-(y_{\sigma}^{[1]})'-\sigma y_{\sigma}^{[1]}-\sigma^2 y$. Hence, we have the operator $\mathbf{S}_q: \mathscr{D}_2(\mathbf{S}_q)\subset L^2(J)\rightarrow L^2(J)$. According to \cite{shkalikov1}, this is the maximal operator associated with the differential expression $-(y_{\sigma}^{[1]})'-\sigma y_{\sigma}^{[1]}-\sigma^2 y$. Several standard properties of equation \eqref{eq:Schrodinger1} were established in \cite{shkalikov1}. For instance, a proper definition for a Cauchy problem with initial data at $x_0\in \overline{J}$ is given in terms of $y(x_0)$ and $y_{\sigma}^{[1]}(x_0)$. For every $g\in L^2(J)$, the Cauchy problem consisting of $\mathbf{S}_qy=\lambda y+g$ with the initial conditions $y(x_0)=a_0$, $y_{\sigma}^{[1]}(x_0)=a_1$, admits a unique solution $y\in \mathscr{D}_2(\mathbf{S}_q)$.  The space of solutions of \eqref{eq:Schrodinger1} is two-dimensional, and two solutions $y,u$ are linearly independent iff their $\sigma$-Wronskian
	\[
	W_{\sigma}[y,v]:= yv_{\sigma}^{[1]}-vy_{\sigma}^{[1]}
	\]
	does not vanish at least at one point of $\overline{J}$ (the proof of these facts can be found in \cite{shkalikov1}).
	
	\section{New regularization based on Polya factorization}
	
	Now, we introduce a new type of regularization for the operator $\mathbf{S}_q$, based on a factorization in terms of a nonvanishing solution. Let $f\in \mathscr{D}_2(\mathbf{S}_q)$ be a solution of $f''-qf=0$ such that $f(x)\neq 0$ for all $x\in \overline{J}$. We call such a function $f$ a {\it nonvanishing solution} of the operator $\mathbf{S}_q$. A natural question arises: can we write $q=\frac{f''}{f}$ in the sense of distributions? To answer this, we first analyze the distribution $\frac{f''}{f}$. Since $f$ does not vanish on $\overline{J}$, a direct computation shows that $f(x)=f(x_0)e^{\int_{x_0}^x\frac{f'(t)}{f(t)}dt}$ for some $x_0\in \overline{J}$, so that $\frac{1}{f}\in AC(\overline{J})$ and $\left(\frac{1}{f}\right)'=-\frac{f'}{f^2}\in L^2(J)$. Thus, $\frac{1}{f}\in W^{1,2}(J)$ and the action of the distribution $\frac{f''}{f}$ is given by
	\begin{equation}\label{eq:distributionf}
		\left(\frac{f''}{f}\bigg{|}\phi\right)_{W_0^{1,2}(J)}=\left(f''\bigg{|}\frac{\phi}{f}\right)_{W_0^{1,2}(J)}=-\int_Jf'\left(\frac{\phi'f-\phi f'}{f^2}\right)=\int_J\left\{ \left(\frac{f'}{f}\right)^2\phi -\left(\frac{f'}{f}\right)\phi'\right\}
	\end{equation}
	Let us denote this distribution by $q_f$. Hence
	\begin{align*}
		(f''-fq_f|\phi)_{W_0^{1,2}}&=\int_J\left\{-f'\phi'-\left(\frac{f'}{f}\right)^2(f\phi)+\frac{f'}{f}(f\phi)'\right\}\\
		&= \int_J\left\{-f'\phi'-\frac{(f')^2}{f}\phi +\frac{(f')^2}{f}\phi +f'\phi'\right\}=0
	\end{align*}
	Thus, $fq=fq_f$ as distributions. For every $\phi\in W_0^{1,2}(J)$, we have $\frac{\phi}{f}\in W_0^{1,2}(J)$, and therefore
	\[
	(q|\phi)_{W_0^{1,2}(J)}=\left(fq\bigg{|}\frac{\phi}{f}\right)_{W_0^{1,2}(J)}=\left(fq_f\bigg{|}\frac{\phi}{f}\right)_{W_0^{1,2}(J)}=(q_f|\phi)_{W_0^{1,2}(J)}
	\]
	Hence $q=q_f$ as distributions. Consequently,
	\[
	\int_J\sigma \phi' = \int_J\left\{\frac{f'}{f}\phi'-\left(\frac{f'}{f}\right)^2\phi\right\}.
	\]
	Therefore, we can express $\sigma$ in terms of $\frac{f'}{f}$ and an anti-derivative of $\left(\frac{f'}{f}\right)^2$. For convenience, we normalize $f$ as follows: given $x_0\in \overline{J}$, we assume that $f(x_0)=1$. 
	Hence $q=q_f$ as distributions. Thus, we can choose
	\begin{equation}\label{eq:sigmaf}
		\sigma_f(x)=\frac{f'(x)}{f(x)}+\int_{x_0}^x\left(\frac{f'(t)}{f(t)}\right)^2dt
	\end{equation}
	as an anti-derivative of $q_f$.
	
	\begin{definition}
		A potential $q\in W^{-1,2}(J)$ is called {\bf admissible} if there exists a nonvanishing solution $f\in \mathscr{D}_2(\mathbf{S}_q)$ of $\mathbf{S}_qf=0$. In such a case, we denote $q=q_f$, choose $\sigma_f$ as the corresponding anti-derivative, and denote the operator $\mathbf{S}_q$ by $\mathbf{S}_f$.
	\end{definition}
	
	\begin{theorem}\label{Th:maintheoremadmissible}
		Every distribution $q\in W^{1,2}(J)$ is admissible.
	\end{theorem}
   The proof of this theorem is presented in Section \ref{Sec:nonvanish}. 
	
	\begin{remark}\label{remark:bijection}
		Since $\frac{1}{f}\in W^{1,2}(J)$, it follows that the map $W^{1,2}(J)\ni y\mapsto \frac{y}{f}\in W^{1,2}(J)$ is a homeomorphism. Indeed, a direct computation shows that for every $y\in W^{1,2}(J)$,
		\[
		\|fy\|_{W^{1,2}(J)}\leq \sqrt{3C}\|f\|_{W^{1,2}(J)}\|y\|_{W^{1,2}(J)},
		\]
		where  $C$ denotes the norm of the embedding  $W^{1,2}(J)\hookrightarrow C(\overline{J})$. Similar for $\frac{y}{f}$.
	\end{remark}

	Denote $D=\frac{d}{dx}$. The {\it impedance operator} associated with {\it the proper impedance} $f$ is given by
	\begin{equation}\label{eq:}
		\mathbf{L}_fv=-\frac{1}{f^2}Df^2Dv=-D^2v-\frac{2f'}{f}Dv\qquad \text{for  }v\in W^{2,1}(J).
	\end{equation}
	Note that $\mathbf{L}_f:W^{2,2}(J)\rightarrow L^2(J)$.
	
	\begin{theorem}[Polya factorization]
		The following factorization holds (in the distributional sense):
		\begin{equation}\label{eq:Polyafact}
			\mathbf{S}_fy=-\frac{1}{f}Df^2D\frac{y}{f}\qquad \forall y\in \mathscr{D}_2(\mathbf{S}_f)
		\end{equation}
		or, equivalently,
		\begin{equation}\label{eq:Schrodingertoimpedance}
			\mathbf{S}_fy=f\mathbf{L}_f \left[\frac{y}{f}\right]\qquad \forall y\in \mathscr{D}_2(\mathbf{S}_f).
		\end{equation}
	\end{theorem}
	
	\begin{proof}
		Let $y\in \mathscr{D}_2(\mathbf{S}_f)$ and set $v=\frac{y}{f}\in W^{1,2}(J)$. From equalities \eqref{eq:actionopS} and  \eqref{eq:distributionf}, we get
		\begin{align*}
			(\mathbf{S}_fy|\phi)_{W_0^{1,2}(J)}&=\int_J\left\{(fv)'\phi'+\left(\frac{f'}{f}\right)^2(fv\phi)-\frac{f'}{f}(fv\phi)'\right\}\\
			&= \int_J \left\{f'v\phi'+fv'\phi'+\frac{(f')^2}{f}v\phi -\frac{f'}{f}(f'v\phi+fv'\phi+fv\phi')\right\}\\
			&= \int_J \{fv'\phi'-f'v'\phi\} \\
			&= \int_J \{v'(f\phi)'-2f'v'\phi\}
		\end{align*}
		By Remark \ref{remark:bijection} and the assumption that $\mathbf{S}_fy$ is $L^2$-regular, it follows that $v\in W^{2,1}(J)$. Hence
		
		\begin{align*}
			(\mathbf{S}_fy|\phi)_{W_0^{1,2}(J)}= \int_J\{-v''f\phi-2f'v'\phi\}=\int_J(f\mathbf{L}_fv)\phi.
		\end{align*}
		Since $\mathbf{S}_fy$ is $L^2$-regular, we have $f\mathbf{L}_f\left[\frac{y}{f}\right]\in L^2(J)$ and we obtain \eqref{eq:Polyafact}, as desired.
	\end{proof}
	
	Let us introduce the operators defined for $y\in W^{1,2}(J)$ as follows:
	\begin{equation}\label{eq:operatorDf}
		\mathbf{D}_fy:= fD\frac{y}{f}=y'-\frac{f'}{f}y, \quad \mathbf{D}_{\frac{1}{f}}y:=\frac{1}{f}D(fy)=y'+\frac{f'}{f}y.
	\end{equation}
	
	As a corollary, we obtain the following properties.
	
	\begin{theorem}\label{th:factorizacion}
		The following statements hold:
		\begin{itemize}
			\item[(i)] $y_{\sigma_f}^{[1]}(x)=\mathbf{D}_fy(x)-y(x)\int_{x_0}^x\left(\frac{f'(t)}{f(t)}\right)^2dt\quad$ for $y\in W^{1,2}(J)$.
			\item[(ii)] $y_{\sigma_f}^{[1]}(x_0)=\mathbf{D}_fy(x_0)\quad$ for $y\in W^{1,2}(J)$.
			\item[(iii)] $\mathscr{D}_2(\mathbf{S}_f)=\{y\in W^{1,2}(J)\, |\, \mathbf{D}_fy\in W^{1,2}(J)\}$.
			\item[(iv)] The operator $\mathscr{D}_2(\mathbf{S}_f)\ni y\mapsto \frac{y}{f}\in W^{2,2}(J)$ is a bijection.
			\item[(v)] $\mathbf{S}_f y= -\mathbf{D}_{\frac{1}{f}}\mathbf{D}_fy\quad$ for $y\in \mathscr{D}_2(\mathbf{S}_f)$. 
		\end{itemize}
		
	\end{theorem}
	
	\begin{proof}
		The proofs of (i) and (ii) follow directly from the definitions  \eqref{eq:quasiderivative} and \eqref{eq:sigmaf}.
		
		For (iii), by the Polya factorization and Remark \ref{remark:bijection}, it is clear that if $y\in W^{1,2}(J)$ satisfies $\mathbf{D}_fy\in W^{1,2}(J)$, then $y\in \mathscr{D}_2(\mathbf{S}_f)$. Conversely, suppose that $y\in \mathscr{D}_2(\mathbf{S}_f)$ and denote $g=\mathbf{S}_fy\in L^2(J)$. From (i), we have $\mathbf{D}_fy\in W^{1,1}(J)$, and by the Polya factorization $\frac{1}{f}D[f\mathbf{D}_fy]=-g\in L^2(J)$. Hence $D[f\mathbf{D}_fy]=-fg\in L^2(J)$ which implies $f\mathbf{D}_fy\in W^{1,2}(J)$. So, by Remark \ref{remark:bijection}, we conclude that $\mathbf{D}_fy\in W^{1,2}(J)$.
		
		For (iv), point (iii) shows that for $y\in \mathscr{D}_2(\mathbf{S}_f)$, if we define $v=\frac{y}{f}$, then $fDv=w\in W^{1,2}(J)$. By Remark \ref{remark:bijection}, we obtain  $Dv=\frac{w}{f}\in W^{1,2}(J)$, so $v\in W^{2,2}(J)$. Conversely, if $v\in W^{2,2}(J)$, the function $y=\frac{v}{f}\in W^{1,2}(J)$ satisfies  $(\mathbf{S}_fy|\phi)_{W^{1,2}_0(J)}=\left( f\mathbf{L}_f\left[\frac{y}{f}\right]\bigg{|}\phi \right)_{W^{1,2}_0(J)}$ for all $\phi\in W^{1,2}_0(J)$. Thus, $y\in \mathscr{D}_2(\mathbf{S}_f)$.
		
		Finally, (v) follows from (iii) and \eqref{eq:Polyafact}.
	\end{proof}
	\begin{remark}\label{Remark:cauchyproblem}
		By Theorem \ref{th:factorizacion}(ii), the initial and boundary conditions of $y_{\sigma_f}^{[1]}$ at $x=x_0$ can be reformulated in terms of $\mathbf{D}_fy(x_0)$.
	\end{remark}
	\begin{remark}\label{Remark:initialconditionsunitaryoperator}
		The operator $\mathscr{D}_2(\mathbf{S}_f)\ni y\mapsto \frac{y}{f}\in W^{2,2}(J)$ preserves the initial conditions at $x_0$. Indeed, set $v=\frac{y}{f}$. Since $f(x_0)=1$, hence $y(x_0)=v(x_0)$. From Theorem \ref{th:factorizacion}(ii), we conclude that
		$y_{\sigma_f}^{[1]}(x_0)=v'(x_0)$.
		
		As in the regular case, we call the operator $W^{2,2}(J)\ni v\mapsto fv\in \mathscr{D}_2(\mathbf{S}_f)$ the {\bf Liouville transform}. When $f$ is positive, the Liouville transform is a unitary operator from $L^2(J; f^2(x)dx)$ onto $L^2(J)$ and $\mathbf{L}_f$ and $\mathbf{S}_f$ are unitary equivalent.
	\end{remark}
	\begin{remark}
		From the factorization \eqref{eq:Polyafact} we obtain a solution $f_1$ of the equation $f_1''-qf_1=0$, which is given by
		\begin{equation}\label{eq:Abelsol}
			f_1(x):=f(x)\int_{x_0}^x\frac{dt}{f^2(t)}.
		\end{equation}
		Indeed, note that $\frac{f_1}{f}\in W^{2,2}(J)$ so $f_1\in \mathscr{D}_2(\mathbf{S}_f)$ and by \eqref{eq:Polyafact} it follows that $\mathbf{S}_ff_1=0$. The $\sigma_f$-Wronskian of $f$ and $f_1$ is computed as
		\begin{align*}
			W_{\sigma_f}[f,f_1](x) &= f(x)(f_1'(x)-\sigma_f(x) f_1(x))-(f'(x)-\sigma_f(x)f(x))f_1(x)\\
			&=f(x)f_1'(x)-f'(x)f_1(x)\\
			&= f(x)f'(x)\int_{x_0}^x\frac{dt}{f^2(t)}+1-f'(x)f(x)\int_{x_0}^x\frac{dt}{f^2(t)}=1.
		\end{align*}
		Hence $f$ and $f_1$ are linearly independent. As in the regular case, the solution $f_1$ will be called the {\bf Abel solution} associated with $f$.
	\end{remark}
	Let us consider the operator $\mathbf{R}_f:L^2(J)\rightarrow L^2(J)$ given by
	\begin{equation}\label{eq:operatorR}
		\mathbf{R}_fg(x):=f(x)\int_{x_0}^x\frac{1}{f^2(t)}\left[\int_{x_0}^tf(s)g(s)ds\right]dt,\quad g\in L^2(J).
	\end{equation}
	Note that $\mathbf{R}_f\in \mathcal{B}(L^2(J),W^{1,2}(J))$ and that for $g\in L^2(J)$, $\frac{1}{f}\mathbf{R}_fg\in W^{2,2}(J)$, so $\mathbf{R}_fg\in \mathscr{D}_2(\mathbf{S}_f)$. From \eqref{eq:Polyafact}, $\mathbf{S}_f\mathbf{R}_fg=g$ and $\mathbf{R}_f$ is a bounded right inverse for $\mathbf{S}_f$.
	
	A straightforward computation yields the following result.
	\begin{proposition}
		The unique solution of the Cauchy problem
		\begin{equation}
			\begin{cases}
				\mathbf{S}_fy=g,\\
				y(x_0)=c_0,\; \mathbf{D}_fy(x_0)=c_1,
			\end{cases}
		\end{equation}
		is given by
		\begin{equation}
			y=\mathbf{R}_fg+c_1f_1+c_0f.
		\end{equation}   
	\end{proposition}
	
	\section{The spectral parameter power series}
	
	\begin{definition}
		The {\bf formal powers} associated with $f$ are the functions $\{\varphi_f^{(k)}\}_{k=0}^{\infty}$ defined recursively as follows:
		\begin{align}
			\varphi_f^{(0)}:= &f,\\
			\varphi_f^{(1)}:= &f_1,\\
			\varphi_f^{(k)}:= &k(k-1)\mathbf{R}_f\varphi_f^{(k-2)},\quad k\geq 2,
		\end{align}
	where $\mathbf{R}_f$ is the integral operator defined in \eqref{eq:operatorR}.
	\end{definition}
	\begin{remark}
		The first formal power satisfies the conditions $\varphi_f^{(0)}(x_0)=1$ and $\mathbf{D}_f\varphi_f^{(0)}(x_0)=0$, while the second one $\varphi_f^{(1)}(x_0)=0$ and $\mathbf{D}_f\varphi_f^{(1)}(x_0)=1$. The sub-sequences formal powers satisfy $\varphi_f^{(k)}(x_0)=\mathbf{D}_f\varphi_f^{(k)}(x_0)=0$. Furthermore, $\mathbf{S}_f\varphi_f^{(0)}=\mathbf{S}_f\varphi_f^{(1)}=0$ and 
		\begin{equation}
			\mathbf{S}_f\varphi_f^{(k)}=k(k-1)\varphi_f^{(k-2)},\quad k\geq 2.
		\end{equation}
	\end{remark}
	Let $\lambda=\rho^2$ with $\rho\in \mathbb{C}$, and let $e_f(\rho,x)$ be the unique solution of Eq. \eqref{eq:Schrodinger1} satisfying the initial conditions
	\begin{equation}\label{eq:initiale}
		e_f(\rho,x_0)=1,\quad \mathbf{D}_fe_f(\rho,x_0)=i\rho
	\end{equation}
	The following theorem generalizes the well-known spectral parameter power series representation to the distributional case.
	\begin{theorem}
		The solution $e_f(\rho,x)$ admits the spectral parameter power series (SPPS) representation
		\begin{equation}\label{eq:SPPS}
			e_f(\rho,x)=\sum_{k=0}^{\infty}\frac{(i\rho)^k\varphi_f^{(k)}(x)}{k!}.
		\end{equation}
		The series converges with respect to $x$ in the norm of $W^{1,2}(J)$, and its second derivative series converges in $W^{-1,2}(J)$. The series converges absolutely and uniformly with respect to the parameter $\rho$ on compact subsets of the complex plane.
	\end{theorem}
	\begin{proof}
		As in the case of regular potentials, the formal powers can be rewritten in terms of recursive integrals as follows. Define $X^{(0)}\equiv \tilde{X}^{(0)}\equiv 1$, and for $k\geq 1$ set
		\begin{align*}
			X^{(k)}(x)& =\int_{x_0}^{x}X^{(k-1)}(s)(f^2(s))^{(-1)^k}ds, \\
			\tilde{X}^{(k)}(x)& =\int_{x_0}^{x}\tilde{X}^{(k-1)}(s)(f^2(s))^{(-1)^{k-1}}ds.
		\end{align*}
		Hence 
		\[
		\varphi_f^{(k)} =k!f\cdot \begin{cases}
			\tilde{X}^{(k)}, & \text{if } k \text{ is even},\\
			X^{(k)}, & \text{if } k \text{ is odd}.
		\end{cases}
		\]
		In this way, the SPPS series \eqref{eq:SPPS} can be re-written as (see, e.g., \cite{spps} or  \cite{sppscampos})
		\begin{equation}\label{eq:auxiliarseries1}
			\sum_{k=0}^{\infty}\frac{(i\rho)^k\varphi_f^{(k)}}{k!}=\sum_{k=0}^{\infty}(-1)^k\rho^{2k}f\tilde{X}^{(2k)}+i\rho  \sum_{k=0}^{\infty}(-1)^k\rho^{2k}fX^{(2k+1)}
		\end{equation}
		The following estimates hold:
		\begin{equation}\label{eq:estimatestildeX}
			|\tilde{X}^{(2k)}(x)|,|X^{(2k)}(x)|\leq \frac{C1^kC_2^k}{(k!)^2}, \; |\tilde{X}^{(2k-1)}(x)|\leq \frac{C_1^{k-1}C_2^k}{(k-1)!k!}, \;\; |X^{(2k-1)}(x)|\leq \frac{C_1^kC_2^{k-1}}{k!(k-1)!},
		\end{equation}
		where $C_1=\|f^{-2}\|_{L^{\infty}(J)}$ and $C_2=\|f^2\|_{L^{\infty}(J)}$ (see \cite[Prop. 5]{sppscampos}). Since $D\tilde{X}^{(2k)}=f'\tilde{X}^{(2k)}+\frac{\tilde{X}^{(2k-1)}}{f}$, the above estimates shows that the series $\sum_{k=0}^{\infty}(-1)^k\rho^{2k}f\tilde{X}^{(2k)}$ converges with respect to the variable $x$ in $W^{1,2}(J)$ and uniformly in $\rho$ on every compact subset of $\mathbb{C}$. Similar for the second series. Consequently, the series of the distributional second derivative of both series converges in $W^{-1,2}(J)$. Furthermore, note that
		\[
		\mathbf{D}_f(f\tilde{X}^{(2k)})=\frac{1}{f}\tilde{X}^{(2k-1)},
		\]
		and by the estimates \eqref{eq:estimatestildeX}, $\mathbf{D}_f\sum_{k=0}^{\infty}(-1)^k\rho^{2k}f\tilde{X}^{(2k)}$ belongs to $W^{1,2}(J)$, that is, $\sum_{k=0}^{\infty}(-1)^k\rho^{2k}f\tilde{X}^{(2k)}$ belongs to $\mathscr{D}_2(\mathbf{S}_f)$. The fact that \eqref{eq:auxiliarseries1} satisfies \eqref{eq:Schrodinger1} follows from the relations
		\[
		\mathbf{S}_f\tilde{X}^{(2k)}=\tilde{X}^{(2k-2)},
		\]
		and satisfies the initial conditions of $e_f(\rho,x)$ by construction.
	\end{proof}

	We also introduce the solutions $C_f(\rho,x)$ and $S_f(\rho,x)$ satisfying the initial conditions
	\begin{align}
		C_f(\rho,x_0)=1, & \quad \mathbf{D}_fC_f(\rho,x_0)=0,\label{eq:cosineinitial}\\
		S_f(\rho,x_0)=0, & \quad \mathbf{D}_fS_f(\rho,x_0)=1.\label{eq:sineinitial}
	\end{align}	
	Applying the uniqueness of the Cauchy problem, a direct computation yields the relations
	\begin{align}
		C_f(\rho,x)=\frac{e_f(\rho,x)+e_f(-\rho,x)}{2},  & \quad  S_f(\rho,x)=\frac{e_f(\rho,x)-e_f(-\rho,x)}{2i\rho},\label{eq:sineasexponential}\\ e_f(\rho,x)=C_f(\rho,x)+i\rho S_f(\rho,x).\label{eq:exponentialtrigonometric}
	\end{align}
	As a corollary of \eqref{eq:sineasexponential}, we obtain the SPPS representation for the solutions $C_f(\rho,x)$ and $S_f(\rho,x)$. The proof is straightforward. 
	\begin{corollary}
		The solutions $C_f(\rho,x)$ and $S_f(\rho,x)$ admit the SPPS representation
		\begin{equation}\label{eq:sppscosinesine}
			C_f(\rho,x)= \sum_{k=0}^{\infty}\frac{(-1)^k\rho^{2k}\varphi_f^{(2k)}(x)}{(2k)!}, \quad S_f(\rho,x)= \sum_{k=0}^{\infty}\frac{(-1)^k\rho^{2k}\varphi_f^{(2k+1)}(x)}{(2k+1)!}.
		\end{equation}
		Each series converges with respect to the variable $x$ in the norm of $W^{1,2}(J)$. The corresponding series of second derivatives converges in $W^{-1,2}(J)$. Moreover, both series converge absolutely and uniformly with respect to the parameter $\rho$ on every compact subset of the complex plane.
	\end{corollary}
	
	\begin{remark}\label{eq:remarkevenfunctionsinrho}
		From the SPPS series \eqref{eq:sppscosinesine}, we observe that, for $x\in J$ fixed, $C_f(\rho,x)$ and $S_f(\rho,x)$ are even functions in the parameter $\rho$. Moreover, both solutions are entire functions in the spectral parameter $\lambda=\rho^2$.
	\end{remark}

	\section{The Darboux transformation}
	Let us consider the operator $\mathbf{S}_{\frac{1}{f}}: \mathscr{D}_2(\mathbf{S}_{\frac{1}{f}}) \rightarrow L^2(J)$ with domain $\mathscr{D}_2(\mathbf{S}_{\frac{1}{f}}):=\{y\in W^{1,2}(J)\,|\, \mathbf{D}_{\frac{1}{f}}y\in W^{1,2}(J)\}$ and whose action is given by
	\[
	\mathbf{S}_{\frac{1}{f}}y:= -\mathbf{D}_f\mathbf{D}_{\frac{1}{f}}y\qquad \text{for   } y\in \mathscr{D}_2(\mathbf{S}_{\frac{1}{f}}).
	\]
	Note that $\frac{1}{f}$ is a non-vanishing solution of the equation $\mathbf{S}_{\frac{1}{f}}y=0$ satisfying normalization $\frac{1}{f(x_0)}=1$. Hence $\mathbf{S}_{\frac{1}{f}}$ is precisely the Schr\"odinger operator with potential $q_{\frac{1}{f}}=f\left(\frac{1}{f}\right)''$.  In the distributional sense,
	\begin{equation}\label{eq:Darbouxtransformpotencial}
		q_{\frac{1}{f}}=2\left(\frac{f'}{f}\right)^2-q_f.
	\end{equation}
	Indeed, note that
	\begin{align*}
		\left(f\left(\frac{1}{f}\right)''\bigg{|}\phi\right)_{W^{1,2}_0(J)}&=\int_J\frac{f'}{f^2}(f\phi'+f'\phi)=\int_J \left\{ \frac{f'}{f}\phi'+\left(\frac{f'}{f}\right)^2\phi\right\}\\
		&=\int_J \left[2\left(\frac{f'}{f}\right)^2\phi +\sigma_f\phi'\right]=\left( 2\left(\frac{f'}{f}\right)^2-q_f\bigg{|}\phi\right)_{W_0^{1,2}(J)}.
	\end{align*}
	\begin{definition} 
		The potential $q_{\frac{1}{f}}$ is called the {\bf Darboux transformation} of $q_f$, and $\mathbf{S}_{\frac{1}{f}}$ is the {\bf Darboux operator} associated with $\mathbf{S}_f$.
	\end{definition} 
	
	Let $y\in \mathscr{D}_2(\mathbf{S}_f)$ be a solution of $\mathbf{S}_fy=\lambda y$. From the factorization $\mathbf{S}_f=-\mathbf{D}_{\frac{1}{f}}\mathbf{D}_f$, we consider the function
	\begin{equation}\label{eq:Darbouxtransform}
		v:= \mathbf{D}_fy.
	\end{equation}
	Note that $\mathbf{D}_{\frac{1}{f}}v=-\mathbf{S}_fy=-\lambda y\in \mathscr{D}_2(\mathbf{S}_f)$. Hence $$\mathbf{S}_{\frac{1}{f}}v=-\mathbf{D}_f\mathbf{D}_{\frac{1}{f}}v=\lambda \mathbf{D}_fy=\lambda v.$$
	Therefore, the operator $\mathbf{D}_f$ transforms solutions of the equation $\mathbf{S}_fy=\lambda y$ into solutions of $\mathbf{S}_{\frac{1}{f}}v=\lambda v$.  Note that $v(x_0)=\mathbf{D}_fy(x_0)$ and $\mathbf{D}_{\frac{1}{f}}v(x_0)=-\lambda y(x_0)$. In the particular case when $y$ is the solution $e_f(\rho,x)$, its Darboux transform $v$ satisfies the initial conditions $v(x_0)=i\rho$ and $\mathbf{D}_fv(x_0)=-\lambda=(i\rho)^2$. By the uniqueness of the Cauchy problem, we have the relation
	\begin{equation}\label{eq:Darbouxtransforme}
		\mathbf{D}_fe_f(\rho,x)=i\rho e_{\frac{1}{f}}(\rho,x).
	\end{equation}
	
	Similarly, we deduce the relations
	\begin{equation}\label{eq:Darbouxtransformsinecosine}
		\mathbf{D}_fC_f(\rho,x)=-\lambda S_{\frac{1}{f}}(\rho,x), \qquad 	\mathbf{D}_fS_f(\rho,x)=C_{\frac{1}{f}}(\rho,x).
	\end{equation}

	\begin{proposition}\label{Prop:Darbouxformalpowers}
		The formal powers satisfy the following relations:
		\begin{align}
			\mathbf{D}_f\varphi_f^{(0)}&=0, \nonumber \\
			\mathbf{D}_f\varphi_f^{(k)}&=k\varphi_{\frac{1}{f}}^{(k-1)} \text{ for all } k\in \mathbb{N}. \label{eq:darbouxformal}
		\end{align}
	\end{proposition}	
	
	\begin{proof}
		Equalities $\mathbf{D}_f\varphi_f^{(0)}=0$ and $\mathbf{D}_f\varphi_f^{(1)}=\frac{1}{f}=\varphi_{\frac{1}{f}}^{(1)}$ follows from the definitions. In order to obtain \eqref{eq:darbouxformal} for $k\geq 2$, we consider the SPPS representation of $e_f(\rho,x)$ and $e_{\frac{1}{f}}(\rho,x)$ together with the relation \eqref{eq:Darbouxtransforme} to obtain:
		\[
		\sum_{k=0}^{\infty}\frac{(i\rho)^{k+1}\varphi_{\frac{1}{f}}^{(k)}}{k!}= i\rho e_{\frac{1}{f}}(\rho,x)=\mathbf{D}_fe_f(\rho,x)=\sum_{k=1}^{\infty}\frac{(i\rho)^k\mathbf{D}_f\varphi_{f}^{(k)}}{k!}.
		\]
		The exchange of the series with the operator $\mathbf{D}_f$ is due to the convergence in $W^{1,2}(J)$. Since $\mathbf{D}_f\varphi_f^{(1)}=\varphi_{\frac{1}{f}}^{(0)}$, reordering indices we obtain
		\[
		\sum_{k=2}^{\infty}\frac{(i\rho)^k\mathbf{D}_f\varphi_{f}^{(k)}}{k!}=\sum_{k=2}^{\infty}\frac{(i\rho)^{k}\varphi_{\frac{1}{f}}^{(k-1)}}{(k-1)!},
		\]
		and comparing as Taylor series in $i\rho$ we conclude \eqref{eq:Darbouxtransformpotencial}.
	\end{proof}

	\section{Integral transmutation operators}
		Through this section, we focus on the case when $J=(0,\ell)$ and $x_0=0$. Let us denote $\mathcal{T}_{\ell}:=\{(x,t)\in \mathbb{R}^2\, |\, 0\leq x\leq \ell, |t|\leq x\}$ (see Figure \ref{fig:triangulos}).
	\begin{theorem}\label{Th:integralrepresentation1}
		There exists a function $K_f\in L^2(\mathcal{T}_{\ell})$ such that
		\begin{equation}\label{eq:exponentialintegral}
			e_f(\rho,x)=e^{i\rho x}+\int_{-x}^xK_f(x,t)e^{i\rho t}dt,\quad 0<x<\ell.
		\end{equation}
	\end{theorem}
	\begin{proof}
		Let $y\in \mathscr{D}_2(\mathbf{S}_f)$ be a solution of $\mathbf{S}_fy=\lambda y$. Define $v_1=-\rho y$ and $v_2=\mathbf{D}_f y$. Using the Polya factorization \eqref{eq:Polyafact} together with \eqref{eq:operatorDf}, a direct computation shows that $(u_1,u_2)^T$ is a solution of the Dirac-type system
		\begin{equation}\label{eq:dirac1}
			\begin{pmatrix}
				0 & 1\\
				-1 & 0
			\end{pmatrix}\begin{pmatrix}
				v_1'\\
				v_2'
			\end{pmatrix}+\begin{pmatrix}
				0 & \frac{f'}{f}\\
				\frac{f'}{f} & 0
			\end{pmatrix} \begin{pmatrix}
				v_1\\
				v_2
			\end{pmatrix}=\rho \begin{pmatrix}
				v_1\\
				v_2
			\end{pmatrix},\quad \begin{pmatrix}
				v_1(0)\\
				v_2(0)
			\end{pmatrix}= \begin{pmatrix}
				-\rho y(0)\\
				\mathbf{D}_fy(0)
			\end{pmatrix}.
		\end{equation}

		We consider the general matrix system
		\begin{equation}\label{eq:Dirac2}
			BU'(x)+Q_f(x)U(x)=\rho U(x),\quad 0<x<\ell,
		\end{equation}
		where
		\[
		B=\begin{pmatrix}
			0&-1\\
			1&0
		\end{pmatrix} \quad \text{and}\quad Q_f=\begin{pmatrix}
			0 & \frac{f'}{f}\\
			\frac{f'}{f} & 0
		\end{pmatrix}.
		\]
		In this case, $Q_f\in L^2((0,\ell); \mathbb{C}^{2\times 2})$, meaning that $(Q_f)_{i,j}\in L^2(0,\ell)$ for $i,j=1,2$. According to \cite[Ch. 2, Theorems 2.1 and 2.3]{haratumyant}, there exists a matrix function $G^f(x,t)$ defined on $\mathcal{T}_{\ell}$ with the property that for $x\in (0,\ell]$ fixed, the function $G^f(x,\cdot)\in L^2((-x,x);\mathbb{C}^2)$ and the unique matrix solution of \eqref{eq:Dirac2} $U(\rho,x)$ satisfying the initial condition $U(\rho,0)=I_{2\times 2}$, admits the integral representation
		\begin{equation}
			U(\rho,x)=e^{-B\rho x}+\int_{-x}^{x}G^f(x,t)e^{-B\rho t}dt,
		\end{equation}
		where $e^{-B\rho x}=\begin{pmatrix}
			\cos(\rho x) & -\sin(\rho x) \\
			\sin(\rho x) & \cos(\rho x)
		\end{pmatrix}$.
		
		Consider the case when $y$ is the solution $S_f(\rho,x)$. Then $(v_1,v_2)^T=(-\rho S_f(\rho,x),\mathbf{D}_fS_f(\rho,x))^T$ is the unique solution of \eqref{eq:dirac1} satisfying the initial condition $(v_1(0),v_2(0))^T=(0,1)^T$. By the theory of linear systems, we have
		\begin{equation}\label{eq:auxiliar0}
			\begin{pmatrix}
				v_1\\
				v_2
			\end{pmatrix}= U\begin{pmatrix}
				0\\
				1
			\end{pmatrix}= \begin{pmatrix}
				-\sin(\rho x)\\
				\cos(\rho x)
			\end{pmatrix}+\int_{-x}^{x}G^f(x,t)\begin{pmatrix}
				-\sin(\rho t)\\
				\cos(\rho t)
			\end{pmatrix}dt.
		\end{equation}
		Consequently,
		\[
		-\rho S_f(\rho,x)= u_{1,2}(\rho,x)= -\sin(\rho x)+\int_{-x}^{x}\left(G_{1,2}^f(x,t)\cos(\rho t)-G_{1,1}^f(x,t)\sin(\rho t)\right)dt.
		\]
		Thus,
		\begin{equation}\label{eq:auxiliar1}
			\rho S_f(\rho,x)-\sin(\rho x)-\int_{-x}^{x}G_{1,1}^f(x,t)\sin(\rho t)dt=-\int_{-x}^xG_{1,2}^f(x,t)\cos(\rho t)dt.
		\end{equation}
		By Remark \ref{Remark:initialconditionsunitaryoperator}, the left-hand side of \eqref{eq:auxiliar1} is an odd function of $\rho$, whereas the right-hand side is even. Hence, both sides must equal zero, and we conclude that
		\begin{equation}\label{eq:sineintegral}
			S_f(\rho,x)=\frac{\sin(\rho x)}{\rho}+\int_{-x}^{x}G_{1,1}^f(x,t)\frac{\sin(\rho t)}{\rho}dt.
		\end{equation}
		
		Now, consider the solution $C_f(\rho,x)$. By \eqref{eq:Darbouxtransformsinecosine}, we have $C_f(\rho,x)=\mathbf{D}_{\frac{1}{f}}S_{\frac{1}{f}}(\rho,x)$. Applying the same procedure as before, but now for the system with $Q_{\frac{1}{f}}$, we obtain the corresponding kernel $G^{\frac{1}{f}}$. Setting $(v_1,v_2)^T=(-\rho S_{\frac{1}{f}}, \mathbf{D}_{\frac{1}{f}}S_{\frac{1}{f}})^T$ as the solution of the system with initial conditions $(0,1)^T$,and using \eqref{eq:auxiliar0} we get
		\[
		C_f(\rho,x)=\cos(\rho x)+\int_{-x}^{x}\left(G_{2,2}^{\frac{1}{f}}(x,t)\cos(\rho t)-G_{2,1}^{\frac{1}{f}}\sin(\rho t)\right)dt.
		\]
		By Remark \ref{Remark:initialconditionsunitaryoperator}, $C_f(\rho,x)$ is an even function of $\rho$, and using the parity we conclude that
		\begin{equation}\label{eq:cosineintegral}
			C_f(\rho,x)=\cos(\rho x)+\int_{-x}^{x}G_{2,2}^{\frac{1}{f}}(x,t)\cos(\rho t)dt.
		\end{equation}
		Finally, from \eqref{eq:exponentialtrigonometric}, we get
		\begin{align*}
			e_f(\rho,x)&=e^{i\rho x}+\int_{-x}^{x}G_{2,2}^{\frac{1}{f}}(x,t)\cos(\rho t)dt+\int_{-x}^{x}G_{1,1}^f(x,t)i\sin(\rho t)dt\\
			&= e^{i\rho x}+\int_{-x}^{x}G_{2,2}^{\frac{1}{f}}(x,t)\left(\frac{e^{i\rho t}+e^{-i\rho t}}{2}\right)dt+\int_{-x}^{x}G_{1,1}^f(x,t)\left(\frac{e^{i\rho t}-e^{-i\rho t}}{2}\right)dt\\
			&=e^{i\rho x}+\int_{-x}^{x}\frac{G_{2,2}^{\frac{1}{f}}(x,t)+G_{1,1}^f(x,t)}{2}e^{i\rho t}dt+\int_{-x}^{x}\frac{G_{2,2}^{\frac{1}{f}}(x,t)-G_{1,1}^f(x,t)}{2}e^{-i\rho t}dt
		\end{align*}

		Applying the change of variables $t\mapsto -t$ in the second integral, we conclude that
		
		\[
		e_f(\rho,x)= e^{i\rho x}+\int_{-x}^{x}\left\{\frac{G_{2,2}^{\frac{1}{f}}(x,t)+G_{2,2}^{\frac{1}{f}}(x,-t)}{2}+\frac{G_{1,1}^f(x,t)-G_{1,1}^f(x,-t)}{2}\right\}e^{i\rho t}dt.
		\]
		
		If we denote by $P^{\pm}_tg(t):=\frac{g(t)\pm g(-t)}{2}$ the even and odd projections in $L^2(-x,x)$, we obtain \eqref{eq:exponentialintegral} with
		\begin{equation}\label{eq:kernelexponential}
			K_f(x,t):=P_t^+G_{2,2}^{\frac{1}{f}}(x,t)+P_t^{-}G_{1,1}^f(x,t).
		\end{equation}
		Finally, if we denote $d(x)=\int_0^x\left|\frac{f'(t)}{f(t)}\right|^2dt$, then according to \cite[p. 40, Eq. (2. 44)]{haratumyant}, the following estimates hold:
		\begin{equation*}
			\int_{-x}^{x}|G_{i,i}^{f^{(-1)^j}}(x,t)|^2dt\leq d(x)+2d^2(x)[x^2d(x)+x]e^{xd(x)},\quad 0<x<\ell, i=1,2, j=0,1.
		\end{equation*}
		Note that the right-hand side is uniformly bounded in $x$ by $d_f+2d_f^2(\ell^2d_f+\ell)e^{\ell d_f}$, where $d_f=\left\| \frac{f'}{f}\right\|_{L^2(0.\ell)}^2$. Consequently,
		\begin{equation}\label{eq:estimateintegralK}
			\iint_{\mathcal{T}_{\ell}}|G_{i,i}^{f^{(-1)^j}}|^2 = \int_0^{\ell}dx\int_{-x}^{x}|G_{i,i}^{f^{(-1)^j}}(x,t)|^2dt\leq \ell (d_f+2d_f^2(\ell^2d_f+\ell)e^{\ell d_f}).
		\end{equation}
		Thus, $G_{2,2}^{\frac{1}{f}}, G_{2,2}^f\in L^2(\mathcal{T}_{\ell})$, and consequently, $K_f\in L^2(\mathcal{T}_{\ell})$, as desired.
	\end{proof}
	\newline

	Consider the operator  defined by
	\[
	\mathbf{T}_fy(x)=y(x)+\int_{-x}^{x}K_f(x,t)y(t)dt, \quad y\in L^2(-\ell,\ell).
	\]
	Since $K_f\in L^2(\mathcal{T}_{\ell})$, it follows that $\mathbf{T}_f\in \mathcal{B}(L^2(-\ell,\ell), L^2(0,\ell))$. We then obtain the relation
	\begin{equation}\label{eq:transmutationexponential}
		e_f(\rho,x)=\mathbf{T}_f[e^{i\rho x}],
	\end{equation}
	and from the equalities \eqref{eq:sineasexponential}, we deduce
	\begin{equation}\label{eq:transmutationcosinesine}
		C_f(\rho,x)=\mathbf{T}_f[\cos(\rho x)], \quad S_f(\rho,x)=\mathbf{T}_f\left[\frac{\sin(\rho t)}{\rho}\right].
	\end{equation}
	Moreover, using the parity in the variable $t$ of $\cos(\rho t)$ and $\frac{\sin(\rho t)}{\rho}$ together with the definition \eqref{eq:kernelexponential} of $K_f(x,t)$, one can recover the integral representations \eqref{eq:sineintegral} and \eqref{eq:cosineintegral}.
	
	\begin{remark}
		In the case when $q_f\in L^2(0,\ell)$, it is known that the kernel $K_f$ satisfies the Goursat conditions
		\[
		K_f(x,x)=\frac{f'(0)}{2}+\frac{1}{2}\int_0^xq_f(s)ds,\quad K_f(x,x)=\frac{f'(0)}{2}
		\]
		Note that $2K_f(x,x)$ is an antiderivative of $q$ satisfying the initial condition $2K(0,0)=f'(0)$. Comparing with the definition \eqref{eq:sigmaf}, we conclude that
		\begin{equation}\label{eq:traceofK}
			K_f(x,x)=\frac{1}{2}\sigma_f(x).
		\end{equation} 
		In the general case when $q_f\in W^{-1,2}(0,\ell)$ is admissible, we observe that
		\[
		2K_f(x,x)=G_{2,2}^{\frac{1}{f}}(x,x)+G_{2,2}^{\frac{1}{f}}(x,-x)+G_{1,2}^f(x,-x)-G_{1,2}^f(x,-x).
		\]
		According to \cite[p. 33]{haratumyant}, $G^f(x,x)=\frac{1}{2}\left[\frac{f'(x)}{f(x)}\begin{pmatrix}
			1 & 0\\
			0 & -1
		\end{pmatrix}+\int_0^x\left(\frac{f'(s)}{f(s)}\right)^2ds\cdot I_{2\times 2}\right]$. In particular, $G_{1,1}^f(x,x)= \frac{f'(x)}{f(x)}+\int_0^x\left(\frac{f'(s)}{f(s)}\right)^2=G_{2,2}^{\frac{1}{f}}(x,x)$. Hence $$2K_f(x,x)= \sigma_f(x)+G_{2,2}^{\frac{1}{f}}(x,-x)-G_{1,1}^f(x,-x).$$
		For the characteristic $(x,-x)$, it is established in \cite[p. 34]{haratumyant} that, if $f'(0)$ is well defined, then $G^f(x,-x)=\frac{f'(0)}{2}\begin{pmatrix}
			1 & 0\\
			0 & -1
		\end{pmatrix}$, so $G_{1,1}^f(x,x)=\frac{f'(0)}{2}$ and $G_{2,2}^{\frac{1}{f}}(x,-x)=\frac{f'(0)}{2}$. Hence, if $f'(0)$ is well defined, then formula \eqref{eq:traceofK} remains valid.
	\end{remark}
	
	An interesting point is that we can determine the explicit action of the operator $\mathbf{T}_f$ on a dense set.
	\begin{theorem}\label{Th:mappingproperty}
		The following relations hold:
		\begin{equation}\label{eq:mappingprop}
			\mathbf{T}_f[x^k]=\varphi_f^{(k)}(x)\quad \forall k\in \mathbb{N}_0.
		\end{equation}
	\end{theorem}
	\begin{proof}
		From the equality \eqref{eq:transmutationexponential}, and since the series $e^{i\rho t}=\sum_{k=0}^{\infty}\frac{(i\rho)^kx^k}{k!}$ converges uniformly on $[-\ell,\ell]$, the continuity of the operator $\mathbf{T}_f$ yields
		\[
		e_f(\rho,x)=\sum_{k=0}^{\infty}\frac{(i\rho)^k\mathbf{T}_f[x^k]}{k!}.
		\]
		Comparing this Taylor series in $i\rho$ with the SPPS representation \eqref{eq:SPPS}, we conclude \eqref{eq:mappingprop}.
	\end{proof}
	\begin{remark}\label{eq:remarkinitialcondition}
		If we consider $u\in C[-\ell,\ell]$, then $(\mathbf{T}_fu)(0)=u(0)$ (the operator preserves the initial condition at $x=0$)
	\end{remark}
	\begin{remark}
		The operator $\mathbf{T}_f: L^2(-\ell,\ell)\rightarrow L^2(0,\ell)$ is surjective.
		
		Indeed, let $y\in L^2(0,\ell)$. Denote by $\tilde{y}$ and $\tilde{K}_f$ the trivial extensions of $y$ and $K_f$ to $(-\ell,\ell)$ and $(-\ell,\ell)^2$, respectively. Since $\tilde{K}_f\in L^2((-\ell,\ell)^2)$, the theory of Volterra integral equations (see, e.g., \cite[Ch. X]{kolmogorov}) ensures that
		\begin{equation*}
			\tilde{y}(x)=u(x)+\int_{-x}^{x}\tilde{K}_f(x,t)u(t)dt,\quad -\ell<x<\ell,
		\end{equation*}
		admits a solution $u\in L^2(-\ell,\ell)$. Taking $0<x<\ell$ in the last equality, we get
		\[
		y(x)=u(x)+\int_{-x}^{x}K(x,t)u(t)dt=\mathbf{T}_fu(x).
		\]
	\end{remark}
	Let $\mathcal{P}[-\ell,\ell]$ be the set of polynomial functions in $[-\ell,\ell]$ (which is dense in $L^2(-\ell,\ell)$ and $C[-\ell,\ell]$).
	\begin{proposition}
		The formal powers $\{\varphi_f^{(k)}\}_{k=0}^{\infty}$ form a complete system in $L^2(0,\ell)$.
	\end{proposition}
	\begin{proof}
		Let $y\in L^2(0,\ell)$, and take $u\in L^2(-\ell,\ell)$ with $y=\mathbf{T}_fu$. Given $\varepsilon>0$, choose $p\in \mathcal{P}[-\ell,\ell]$ such that $\left\|u-p\right\|_{L^2(-\ell,\ell)}<\frac{\varepsilon}{\|\mathbf{T}_f\|_{\mathcal{B}(L^2(-\ell,\ell),L^2(0,\ell))}}$. Set $\widehat{p}=\mathbf{T}_fp\in \operatorname{Span}\{\varphi_f^{(k)}\}_{k=0}^{\infty}$. Then
		\begin{align*}
			\|y-\widehat{p}\|_{L^2(0,\ell)}=\|\mathbf{T}_f(u-p)\|_{L^2(0,\ell)}\leq \|\mathbf{T}_f\|_{\mathcal{B}(L^2(-\ell,\ell),L^2(0,\ell))}\|u-p\|_{L^2(-\ell,\ell)}<\varepsilon.
		\end{align*}
	\end{proof}

	Now we investigate the relationships between the operator $\mathbf{T}_f$, and the one corresponding to the associated Darboux equation $\mathbf{S}_{\frac{1}{f}}$, denoted by $\mathbf{T}_{\frac{1}{f}}$. First, by the mapping property \eqref{eq:mappingprop} and Proposition \ref{Prop:Darbouxformalpowers}, we have $\mathbf{D}_f\mathbf{T}_f[1]=\mathbf{D}_f\varphi_f^{(0)}=0=\mathbf{T}_{\frac{1}{f}}[D1]$, and for $k\geq 1$: 
	\[
	\mathbf{D}_f\mathbf{T}_fx^k =\mathbf{D}_f\varphi_f^{(k)}=k\varphi_{\frac{1}{f}}^{(k-1)}=\mathbf{T}_{\frac{1}{f}}Dx^k.
	\]
	By linearity, we obtain the relation
	\[
	\mathbf{D}_f\mathbf{T}_fp=\mathbf{T}_{\frac{1}{f}}p'\qquad \forall p\in \mathcal{P}[-\ell,\ell].
	\]
	This is equivalent to stating that
	\[
	\mathbf{T}_fp(x)=f(x)\int_0^x\frac{\mathbf{T}_{\frac{1}{f}}p'(t)}{f(t)}dt+cf(x),
	\]
	for some $c\in \mathbb{C}$. However, by Remark \eqref{Remark:initialconditionsunitaryoperator}, we have $c=\mathbf{T}_fp(0)=p(0)$. 
	
	Let us consider the operator 
	\[
	\mathbf{J}_f y(x):= f(x)\int_0^x\frac{y(t)}{f(t)}dt,\quad y\in L^2(0,\ell).
	\]
	It is clear that $\mathbf{J}_f\in \mathcal{B}(L^2(0,\ell))$, that $\mathbf{J}_f(L^2(0,\ell))\subset W^{1,2}(0,\ell)$, and that $\mathbf{D}_f\mathbf{J}_f=\mathbf{I}_{L^2(0,\ell)}$. Hence we obtain that
	\begin{equation}\label{eq:polinomialstransmute}
		\mathbf{T}_fp=\mathbf{J}_f\mathbf{T}_{\frac{1}{f}}+p(0)f\qquad \forall p\in \mathcal{P}[-\ell,\ell].
	\end{equation}
	We extend this relation to $W^{1,2}(-\ell,\ell)$.
	
	\begin{theorem}\label{Th:transmdarboux}
		$\mathbf{T}_f(W^{1,2}(-\ell,\ell))\subset W^{1,2}(0,\ell)$ and the following relation holds:
		\begin{equation}\label{eq:transmdarboux}
			\mathbf{D}_f\mathbf{T}_fy=\mathbf{T}_{\frac{1}{f}}y'\qquad \forall y\in W^{1,2}(-\ell,\ell).
		\end{equation}
	\end{theorem}
	\begin{proof}
		Let $y\in W^{1,2}(-\ell,\ell)$. Choose a sequence $\{p_n\}\subset \mathcal{P}[-\ell,\ell]$ such that $p_n\rightarrow y$ uniformly on $[-\ell,\ell]$ and $p_n'\rightarrow y$ in $L^2(-\ell,\ell)$ (see, e.g., \cite[Cor. 23]{mineimpedance3}). Since $\mathbf{T}_f, \mathbf{J}_f\mathbf{T}_{\frac{1}{f}}\in \mathcal{B}(L^2(-\ell,\ell), L^2(0,\ell))$ and $p_n(0)\rightarrow y(0)$, we get
		\begin{align*}
			\mathbf{T}_f y= \lim_{n\rightarrow \infty} 	\mathbf{T}_fp=\lim_{n\rightarrow \infty} \mathbf{J}_f\mathbf{T}_{\frac{1}{f}}p_n'+p_n(0)f=\mathbf{J}_f\mathbf{T}_{\frac{1}{f}}y'+y(0)f.
		\end{align*}
		Consequently, $\frac{1}{f}\mathbf{T}_fy\in W^{1,2}(0,\ell)$, and 
		\[
		\mathbf{D}_f\mathbf{T}_fy= \mathbf{D}_f\mathbf{J}_f\mathbf{T}_{\frac{1}{f}}+y(0)\mathbf{D}_ff=\mathbf{T}_{\frac{1}{f}}y',
		\]
		as desired. By Remark \ref{remark:bijection} it follows that $\mathbf{T}_fy\in W^{1,2}(0,\ell)$.
	\end{proof}
	
	From this transmutation relation, we obtain the so-called transmutation property for the operators $\mathbf{S}_f$ and $-D^2$.
	
	\begin{theorem}\label{Th:transmutationoperator1}
		$\mathbf{T}_f(W^{2,2}(-\ell,\ell))\subset \mathscr{D}_2(\mathbf{S}_f)$ and the following transmutation relation holds:
		\begin{equation}\label{eq:transmutationrelation}
			\mathbf{S}_f\mathbf{T}_fy=-\mathbf{T}_fy''\qquad \forall y\in W^{2,2}(-\ell,\ell).
		\end{equation}
	\end{theorem}
	\begin{proof}
		Let $y\in W^{2,2}(-\ell,\ell)$. By the previous theorem, we have $\mathbf{T}_{\frac{1}{f}}y'\in W^{1,2}(0,\ell)$, and the relation \eqref{eq:transmdarboux} implies that $\mathbf{D}_f\mathbf{T}_fy\in W^{1,2}(0,\ell)$, that is, $\mathbf{T}_fy\in \mathscr{D}_2(\mathbf{S}_f)$. Thus,
		\begin{align*}
			-\mathbf{S}_f\mathbf{T}_fy = \mathbf{D}_{\frac{1}{f}}\left[ \mathbf{D}_f\mathbf{T}_fy\right] =\mathbf{D}_{\frac{1}{f}}\left[ \mathbf{T}_{\frac{1}{f}}y'\right]=\mathbf{T}_fy''. 
		\end{align*}
	\end{proof}
	
	\begin{theorem}\label{th:continuityofoperator}
		$\mathbf{T}_f(C[-\ell,\ell])\subset C[0,\ell]$. Furthermore, $\mathbf{T}_f\in \mathcal{B}(C[-\ell,\ell],C[0,\ell])$.
	\end{theorem}
	\begin{proof}
		According to \cite[p. 32]{haratumyant}, the following estimate holds:
		\[
		\int_{-x}^{x}|G^f(x,t)|dt\leq e^{\|\frac{f'}{f}\|_{L^1(0,\ell)}}-1\leq e^{\sqrt{\ell}\|\frac{f'}{f}\|_{L^2(0,\ell)}}-1
		\] 
		In particular, $\int_{-x}^{x}|K_f(x,t)|dt\leq 4M$, where $M=e^{\sqrt{\ell}\|\frac{f'}{f}\|_{L^2(0,\ell)}}-1$. Then, given $u\in L^{\infty}(-\ell,\ell)$ we obtain
		\[
		|\mathbf{T}_fu(x)|\leq |u(x)|+\int_{-x}^{x}|K_f(x,t)||u(t)|dt\leq (1+M)\|u\|_{L^{\infty}(-\ell,\ell)}.
		\]
		Thus, $\mathbf{T}_f\in \mathcal{B}(L^{\infty}(-\ell,\ell),L^{\infty}(0,\ell))$. Now, given $u\in C[-\ell,\ell]$, take a sequence of polynomials $\{p_n\}\subset \mathcal{P}[-\ell,\ell]$ such that $p_n\rightarrow u$ in $L^{\infty}(-\ell,\ell)$. Hence $\mathbf{T}_fp_n\rightarrow \mathbf{T}_fu$ in $L^{\infty}(0,\ell)$. By Theorem \ref{Th:transmutationoperator1}, $\{\mathbf{T}_fp_n\}\subset C[0,\ell]$, which implies that $\mathbf{T}_fu\in C[0,\ell]$, as desired.
	\end{proof}

	Let us denote by $P^{\pm}: L^2(-\ell,\ell)\rightarrow L^2(-\ell,\ell)$ the even and odd projections. Consider the kernels
	\begin{align}
		K_f^{+}(x,t)&:=\frac{1}{2}\left(K_f(x,t)+K_f(x,-t)\right)=\frac{1}{2}\left(G_{2,2}^{\frac{1}{f}}(x,t)+G_{2,2}^{\frac{1}{f}}(x,-t)\right) \label{eq:evenkernel}\\
		K_f^{-}(x,t)&:=\frac{1}{2}\left(K^(x,t)-K_f(x,-t)\right)=\frac{1}{2}\left(G_{1,1}^f(x,t)+G_{1,1}^f(x,-t)\right),\label{oddkernel}
	\end{align}
	and the corresponding operators $\mathbf{T}^{\pm}_f\in  \mathcal{B}(L^2(0,\ell))$ given by
	\begin{equation}\label{eq:evenoddtransmutations}
		\mathbf{T}_f^{\pm}u(x)=u(x)+\int_0^xK_f^{\pm}(x,t)u(t)dt.
	\end{equation}
	A direct computation shows the following decomposition:
	\begin{equation}\label{eq:decompositionop}
		\mathbf{T}_f =\mathbf{T}_f^+P^+ + \mathbf{T}_f^{-}P^{-}.
	\end{equation}
	From this decomposition, we obtain the following identities:
	\begin{equation}\label{eq:cosineandsinetransm}
		C_f(\rho,x)=\mathbf{T}_f^+[\cos(\rho x)],\qquad S_f(\rho,x)=\mathbf{T}_f^{\pm}\left[\frac{\sin(\rho x)}{\rho}\right].
	\end{equation}
	For this reason, in some contexts $K_f^{+}$ and $K_f^{-}$ are referred to as the {\it cosine and sine transmutation kernels}, respectively. From the decomposition \eqref{eq:decompositionop}, we have
	\begin{equation}
		\mathbf{T}_f^{+}[x^{2k}]=\varphi_f^{(2k)} \quad \text{and}\;\;\; 	\mathbf{T}_f^{-}[x^{2k+1}]=\varphi_f^{(2k+1)}\quad \text{for all } k\in \mathbb{N}_0.
	\end{equation}

	Let us define
	\begin{equation}
		W_+^{2,2}(0,\ell):=\{u\in W^{2,2}(0,\ell)\, |\, u'(0)=0\} \quad \text{and  } \;\;\; W_{-}^{2,2}(0,\ell):=\{u\in W^{2,2}(0,\ell)\, |\, u(0)=0\} 
	\end{equation}
	Both subspaces are closed in $W^{2,2}(0,\ell)$ since they are the preimages of $\{0\}$ under the bounded functionals $\delta,\delta':W^{2,2}(0,\ell)\rightarrow \mathbb{C}$.
	
	\begin{theorem}
		The following relations hold:
		\begin{equation}\label{eq:transrelationevenodd}
			\mathbf{S}_f\mathbf{T}_f^{\pm}u=-\mathbf{T}_f^{\pm}D^2u\qquad \forall u\in W^{2,2}_{\pm}(0,\ell).
		\end{equation}
	\end{theorem}
	\begin{proof}
		Let $E^{\pm}: L^2(0,\ell)\rightarrow L^2(-\ell,\ell)$ be the extension operators defined by 
		\[
		E^{\pm}u(x):=\begin{cases}
			u(x), & 0<x<\ell,\\
			\pm u(-x), & -\ell <x<0.
		\end{cases}
		\]
		It is clear that $P^{\pm}E^{\pm}u=u$ and $P^{\pm}E^{\mp}u=0$ for all $u\in L^2(0,\ell)$. For $u\in W^{1,2}(0,\ell)$, a direct computation shows the following relations (in the distributional way):
		\[
		DE^+u=E^{-}Du \quad \text{and  } \;\;\; DE^{-}u=E^+Du+2u(0)\delta(x).
		\] 
		In particular, if $u(0)=0$, then $DE^{-}u=E^{+}Du$.
		
		Suppose that $u\in W_+^{2,2}(0,\ell)$. By definition, $u'(0)=0$, so
		\[
		D^2E^{+}u=DE^{-}u'=E^+u''\quad \text{and then } E^+u\in W^{2,2}(-\ell,\ell). 
		\]
		Since $\mathbf{T}_fE^+u=\mathbf{T}_f^+u$, it follows that $\mathbf{T}_f^+u\in \mathscr{D}(\mathbf{S}_f)$ and hence
		\[
		\mathbf{S}_f\mathbf{T}_f^+u= \mathbf{S}_f\mathbf{T}_fE^+u= -\mathbf{T}_fD^2E^+u=-\mathbf{T}_fE^+D^2u''=-\mathbf{T}_f^+D^2u.
		\] 
		Similarly, if $u\in W_{-}^{2,2}(0,\ell)$, then $D^2E^{-}u=DE^+u'=E^{-}u''$, so $E^{-}\in W^{2,2}(-\ell,\ell)$. Applying the same argument, we obtain the transmutation relation  \eqref{eq:transrelationevenodd}. 
	\end{proof}
	\begin{remark}
		It is important to emphasize the condition that $\mathbf{T}_f^{\pm}$ are transmutation operators in $W_{\pm}^{2,2}(0,\ell)$, even for regular potentials. For example, in \cite[Example 3]{camposlbases} it was shown that when $q$ is a constant, $\mathbf{S}_f\mathbf{T}_f^{-}[1]\neq 0=\mathbf{T}_f^{-}[D^21]$, because $1\not\in W_{-}^{2,2}(0,\ell)$.
	\end{remark}

	\section{On the existence and construction of the non-vanishing solution}\label{Sec:nonvanish}
	
	In this section,  we consider a general finite interval $J$. Let $q\in W^{-1,2}(J)$ and let $\sigma\in L^2(J)$ be an anti-derivative. We prove Theorem~\ref{Th:maintheoremadmissible}, which establishes the existence of a nonvanishing solution $f \in W^{1,2}(J)$ of
	\begin{equation}\label{eq:equationoff}
		(y_{\sigma}^{[1]})' + \sigma y_{\sigma}^{[1]} + \sigma^2 y = 0.
	\end{equation}
	We begin by considering several particular cases. The first one corresponds to the situation in which the potential $q$ is given by delta point interactions, leading to the following result.
	
	\begin{theorem}
		Every potential given by the perturbation of an $L^2$-regular function by a finite number of delta point interactions is admissible, that is, every potential $q$ of the form
		\begin{equation}
			q(x)=q_{reg}(x)+\sum_{k=1}^{N}\alpha_k\delta(x-x_k),
		\end{equation} 	
		where $q_{reg}\in L^2(0,\ell)$ is in general complex-valued, $\inf J<x_1<\dots <x_n<\sup J$, and $\alpha_1,\dots,\alpha_N\in \mathbb{C}$.
	\end{theorem}
	\begin{proof}
		See \cite{minedeltas}, Prop. 24 and Th. 26, for the proof of existence and for a method of construction.
	\end{proof}
	
	We say that a distribution $q\in W^{-1,2}(J)$ is {\it real-valued} if $\operatorname{Re}(q|\phi)_{W_0^{1,2}(J)}=(q|\operatorname{Re}\phi)_{W_0^{1,2}(J)}$ for all $\phi\in W_0^{1,2}(J)$. Note that $q$ is real-valued iff it admits at least one real-valued antiderivative $\sigma$.

	\begin{theorem}\label{Th:realvaluedpot}
		Every real-valued distributional potential $q$ is admissible.
	\end{theorem}
	\begin{proof}
		Let $\sigma\in L^2(J)$ be a real-valued antiderivative of $q$. Note that Eq. \eqref{eq:equationoff} can be rewritten as a Dirac-type system as follows: Let $u_1=y$ and $u_2=y_{\sigma}^{[1]}$. A straightforward computation shows that $(u_1,u_2)^T$ satisfies the system 
		\begin{equation}\label{eq:systemsigma}
			\begin{pmatrix}
				0 & 1\\
				-1 & 0
			\end{pmatrix}\begin{pmatrix}
				u'_1\\
				u'_2
			\end{pmatrix}= -\begin{pmatrix}
				\sigma^2 & \sigma \\
				\sigma & 1
			\end{pmatrix}\begin{pmatrix}
				u_1\\
				u_2
			\end{pmatrix}.
		\end{equation}
		Denote $\Sigma=\begin{pmatrix}
			\sigma^2 & \sigma \\
			\sigma & 1
		\end{pmatrix}$ and $B=	\begin{pmatrix}
			0 & 1\\
			-1 & 0
		\end{pmatrix}$. Let $U=(U_1 U_2)=\begin{pmatrix}
			u_{1,1} & u_{1,2} \\
			u_{2,1} & u_{2,2}
		\end{pmatrix}$ be the fundamental matrix solution of $BU'=-\Sigma U$ subject to the initial condition $U(x_0) = I_{2\times 2}$ for some $x_0 \in \overline{J}$. Consider
		\begin{equation}
			f=u_{1,1}+iu_{1,2}.
		\end{equation}
		Note that that $f'=f_{\sigma}^{[1]}+\sigma f= u_{2,1}+iu_{2,2}+ \sigma (u_{1,1}+i u_{1,2})\in L^2(J)$ (because $u_{i,j}\in AC(\overline{J})$), so $f\in W^{1,2}(J)$ and $f_{\sigma}^{[1]}= u_{2,1}+iu_{2,2}\in AC(J)$. Since $U$ is solution of the system, by construction $f\in \mathscr{D}_2(\mathbf{S}_q)$ and satisfies $\mathbf{S}_qf=0$. Now, suppose that there exists a point $x_1\in \overline{J}$ such that $f(x_1)=0$. Since $\sigma$ is real-valued, it follows that $U$ is real valued and then $u_{1,1}(x_1)=u_{1,2}(x_1)=0$. Consequently, $\operatorname{det}U(x_1)=0$ and there exists $\alpha\in \mathbb{R}$ such that $U_1(x_1)=\alpha U_2(x_1)$. Define $v := U_1 - \alpha U_2$. Then $v$ satisfies the homogeneous system 
\eqref{eq:systemsigma} and $v(x_1) = 0$, which implies $v \equiv 0$. 
Hence $U_1 \equiv \alpha U_2$, contradicting the linear independence 
of the columns $U_1$ and $U_2$. Therefore, $f$ has no zeros in $\overline{J}$, 
and $f(x_0) = u_{1,1}(x_0) + i u_{1,2}(x_0) = 1$ is the desired 
nonvanishing solution
	\end{proof}
	
	For complex-valued potentials, we extend several techniques from 
\cite{camporesi, sppsnelson, spps} in order to establish the existence 
of a nonvanishing solution $f$
	
	\begin{theorem}
		If the potential $q$ possesses a continuous antiderivative $\sigma$ on $\overline{J}$, then $q$ is admissible.
	\end{theorem}
	\begin{proof}
		As in the proof of  Theorem \ref{Th:realvaluedpot}, the components $u_{1,1}$ and $u_{1,2}$ of the fundamental matrix solution of \eqref{eq:systemsigma} never vanish simultaneously. Since $\sigma\in C(\overline{J})$, we have $u_{1,1},u_{1,2}\in C^1(\overline{J})$. Let $\mathbb{CP}^1$ be the complex projective line, i.e., the quotient of $\mathbb{C}^2\setminus\{(0,0)\}$ under the action of $\mathbb{C}^{*}:=\mathbb{C}\setminus\{0\}$, and let $[a:b]$ denote the equivalence class of the pair $(a,b)$.
		
		Consider the map $F: \overline{J}\rightarrow \mathbb{CP}^1$ given by $F(x)=[u(x): v(x)]$ for $x\in \overline{J}$. The map is well defined and of class $C^1$. According to \cite[Prop. 2.2]{camporesi}, $F$ is never surjective, and the proof is that Sard's theorem implies that $F(\overline{J})$ has measure zero.
		
		Suppose that $[c_1:c_2]\in \mathbb{CP}^1$ is such that $c_1u_{1,1}(\xi)+c_2u_{1,2}(\xi)=0$ for some $\xi\in \overline{J}$. Hence $\begin{vmatrix}
			u_{1,1}(\xi) & u_{1,2}(\xi) \\
			-c_2 & c_1
		\end{vmatrix}=0$, that is, $(u_{1,1}(\xi),u_{1,2}(\xi))$ and $(-c_2,c_2)$ are proportionals, which implies that $[-c_2:c_1]\in F(\overline{J})$. Consequently, the set $C=\{[-c_2:c_1]\in \mathbb{CP}^1\, | \,\exists \xi \in \overline{J} : c_1u_{1,1}(\xi)+c_2u_{1,2}(\xi)=0 \}$ is contained in $F(\overline{J})$, and therefore has measure zero. Hence there exists a pair $[c_1:c_2]\in \mathbb{CP}^1\setminus C$ such that $f=c_1u_{1,2}+c_2u_{1,2}$ has no zeros on $\overline{J}$. Therefore, $f$ is a nonvanishing solution of \eqref{eq:equationoff}. 
	\end{proof}
	\newline 
	
	We can generalize some ideas from the proof of the previous theorems to the general case $q\in W^{-1,2}(J)$ complex-valued. Recall that for $1\leq p\leq \infty$, $W^{1,p}(J)$ is an algebra \cite[Cor. 8.10]{brezis}. Consequently, $|u|^2=u\overline{u}\in W^{1,p}(J)$ for every $u\in W^{1,p}(J)$. Moreover, if $v\in W^{1,p}(J)$ is such that $v(x)\neq 0$ for all $x\in \overline{J}$, then $\frac{1}{v}\in W^{1,p}(J)$.
	\newline 
	
	\begin{proof}[Proof of Theorem \ref{Th:maintheoremadmissible}]
		Denote $u=u_{1,1}$ and $v=u_{1,2}$. Note that  $u'=u_{\sigma}^{[1]}+\sigma u \in L^2(J)$, so $u\in W^{1,2}(J)$. The same argument applies to $v$. Since $u$ and $v$ never vanish simultaneously, we may consider the well-defined map $F(x)=[u(x):v(x)]\in \mathbb{CP}^1$, $x\in \overline{J}$. We recall that $\mathbb{CP}^1$ is homeomorphic to the Riemann sphere $\mathbb{S}^2$ via the map
		\begin{equation}\label{eq:homeomorphismCPS}
			\mathbb{CP}^1\ni [z:w]\mapsto \frac{1}{|z|^2+|w|^2}\left(2\operatorname{Re}(z\overline{w}),2\operatorname{Im}(z\overline{w}), |z|^2-|w|^2\right)\in \mathbb{S}^2.
		\end{equation}
		Composing $F$ with the homeomorphism \eqref{eq:homeomorphismCPS} yields a path $\gamma: \overline{J}\rightarrow \mathbb{S}^2$ given by
		\begin{equation}\label{eq:auxiliarpath}
			\gamma(x):=\frac{1}{|u(x)|^2+|v(x)|^2}\left(2\operatorname{Re}(u(x)\overline{v(x)}),2\operatorname{Im}(u(x)\overline{v(x)}), |u(x)|^2-|v(x)|^2\right).
		\end{equation}
		Since the denominator does not vanish, the function $\gamma$ is continuous on $\overline{J}$, and since $u,v\in W^{1,p}(J)$, it follows that $2\operatorname{Re}(u\overline{v}), \operatorname{Im}(u\overline{v}), |u|^2\pm |v|^2\in W^{1,2}(J)$.  The function $|u|^{2}+|v|^{2}$ is also in $W^{1,2}(J)$ and, by assumption, does not vanish on $\overline{J}$, hence its inverse belongs to $W^{1,2}(J)$. Consequently, each component of the path $\gamma$ belongs to $W^{1,2}(J)$. In particular, $\int_J|\gamma'(x)|dx\leq |J|^{\frac{1}{2}}\left(\int_J|\gamma'(x)|^2dx\right)^{\frac{1}{2}}<\infty$ and $\gamma$ is rectificable. This implies that the  1-dimensional Hausdorff measure of $\gamma(\overline{J})$ is at most $|\gamma(\overline{J})|$ (see \cite[Th. 13]{warner}). Consequently, the Hausdorff dimension of $\gamma(\overline{J})$ is at most $1$ (see \cite[Sec. 2.2]{falconer}). Since the Hausdorff dimension of $\mathbb{S}^2$ is $2$, it follows that $\mathbb{S}^2\setminus \gamma(\overline{J})\neq \emptyset$. Thus, we may choose a point $(p_1,p_2,p_3)\in \mathbb{S}^2\setminus \gamma(\overline{J})$. Let $[c_1:c_2]\in \mathbb{CP}^1$ be its preimage under the homeomorphism  \eqref{eq:homeomorphismCPS}. Hence $[c_1:c_2]\not\in F(\overline{J})$ and, in particular, $c_2u_1-c_1u_2$ does not vanish on $\overline{J}$, from where we obtain the solution $f$.
	\end{proof}
	\begin{corollary}
		If $q$ is defined by a complex Radon measure on $\overline{J}$, then $q$ is admissible.
	\end{corollary}
	\begin{proof}
		Suppose that $(q|\phi)_{W_0^{1,2}(J)}=\int_J\phi d\mu$, where $\mu$ is a complex Radon measure on $\overline{J}$. In this case, $q$ extends to a continuous functional on $C(\overline{J})$, so the Riesz representation theorem implies that $q$ is given in terms of a Riemann-Stieltjes integral $(q|\phi)=\int_J\phi d\sigma$, where $\sigma \in BV(\overline{J})$ (see \cite[Th. 4.4-1]{kreyszig}). Since $\phi \in W^{1,2}_0(J)$ vanish in $\partial J$, we have that $(q|\phi)=-\int_J\phi'\sigma$ (see \cite[Th. 3. 36]{folland}), and $q$ is the distributional derivative of $\sigma \in L^{2}(J)$.
	\end{proof}
	\newline
	
	We end this section by proposing a method to construct a solution $f$, based on the SPPS method for Dirac-type systems of the form $BY'+P(x)Y=\lambda R(x)Y$, developed in \cite{sppsnelson}. In our case, $P\equiv 0_{2\times 2}$, $\lambda=-1$ and $R=\Sigma$.  Applying the method described in \cite[Sec. 2]{sppsnelson}, we obtain that the solutions $u_{1,1}$ and $u_{1,2}$ admit the series representations:
	\begin{equation}
		u_{1,1}=\sum_{k=0}^{\infty}\frac{(-1)^k}{k!}\tilde{X}^{(k)},\qquad u_{1,2}=\sum_{k=0}^{\infty}\frac{(-1)^k}{k!}X^{(k)}
	\end{equation}
	where the functions $\{X^{(k)}\}_{k=0}^{\infty}$ and $\{\tilde{X}^{(k)}\}_{k=0}^{\infty}$ are defined recursively as follows: For $\{X^{(k)}\}_{k=0}^{\infty}$, define $X^{(0)}\equiv 0$ and $Y^{(0)}\equiv 1$, and for $k\geq 1$,
	\begin{align}
		X^{(k)}(x)&=-k\int_{x_0}^x(\sigma(s)X^{(k-1)}(s)+Y^{(k-1)}(s))ds, \nonumber\\
		Y^{(k)}(x)&=k\int_{x_0}^x(\sigma^2(s)X^{(k-1)}(s)+\sigma(s)Y^{(k-1)}(s))ds. \label{eq:interalsXsystem}
	\end{align}
	For $\{\tilde{X}^{(k)}\}_{k=0}^{\infty}$, we set $\tilde{X}^{(0)}\equiv 1$ and $\tilde{Y}^{(0)}\equiv 0$ and apply the same recursion \eqref{eq:interalsXsystem}.  The solution $f$ is obtained by choosing suitable constants $(c_1,c_2)\in \mathbb{C}^2\setminus \{(0,0)\}$ such that $f=c_1u_{1,1}+c_2u_{1,2}$ does not have zeros on $\overline{J}$. When $\sigma$ is real-valued, it suffices to choose $c_1=1$ and $c_2=i$.

	\section{Approximation}
	We will now establish some approximation results for the kernel $K_f$. The following lemma will be useful for later results.
	\begin{lemma}\label{Lemma:convergencefn}
		Let $f\in W^{1,2}(J)$ be a non-vanishing function such that $f(x_0)=1$ for some $x_0\in \overline{J}$. There exists a sequence $\{f_n\}\subset C^{\infty}(\overline{J})$ of non-vanishing functions with $f_n(x_0)=1$ and such that
		\[
		f_n, \frac{1}{f_n}\rightarrow f,\frac{1}{f} \text{ in } W^{1, 2}(J), \qquad  \sigma_{f_n}\rightarrow \sigma_f \; \text{in } L^2(J), \;\;\; \text{and }\;\;\; q_{f_n}\rightarrow q_f \text{ in } W^{-1,2}(J).
		\] 
	\end{lemma}
	\begin{proof}
		Let $\tau_f:=\frac{f'}{f}\in L^2(J)$. Take a sequence $\{\psi_n\}\subset C^{\infty}(\overline{J})$ such that $\psi_n\rightarrow \tau_f$ in $L^2(J)$. Define $f_n(x)=e^{\int_{x_0}^x\psi_n(s)}$. Then each $f_n\in C^{\infty}(\overline{J})$ is non-vanish and satisfies $f_n(x_0)=1$. Since $f(x)=e^{\int_{x_0}^{x}\tau_f(s)ds}$, it is clear that $f_n\rightarrow f$ and $\frac{1}{f_n}\rightarrow \frac{1}{f}$ in $W^{1,2}(J)$.  For the convergence of $\sigma_{f_n}$, observe that
		\begin{align*}
			\int_J\left|\int_{x_0}^x(\psi_n^2(s)-\tau_f^2(s))ds\right|^2dx \leq & \int_J\left(\int_{x_0}^{x}|\psi_n(s)-\tau_f(s)||\psi_n(s)+\tau_f(s)|ds\right)dx\\
			\leq &|J|\|\psi_n-\tau_f\|_{L^2(J)}\|\psi_n+\tau_f\|_{L^2(J)}\leq |J|M\|\phi_n-\tau_f\|_{L^2(J)},
		\end{align*}
		where $M=\sup_{n\in \mathbb{N}}\|\psi_n\|_{L^2(J)}+\|\tau_f\|_{L^2(J)}$. Thus, $\int_{x_0}^x\psi_n^2 \rightarrow \int_{x_0}^x\tau_f^2$ in $L^2(J)$. By \eqref{eq:sigmaf}, $\sigma_{f_n}\rightarrow \sigma_f$ in $L^2(J)$.
		
		Finally, given $\phi\in W^{1,2}(J)$, we have
		\[
		|(q_{f_n}-q_f|\phi)_{W^{1,2}_0(J)}|\leq \int_J|\sigma_{f_n}-\sigma_f||\phi'|\leq \|\sigma_{f_n}-\sigma_f\|_{L^2(J)}\|\phi\|_{W_0^{1,2}(J)}.
		\]
		Consequently, $\|q_{f_n}-q_f\|_{W^{-1,2}(J)}\leq \|\sigma_{f_n}-\sigma_f\|_{L^2(J)}\rightarrow 0$ as $n\rightarrow \infty$.
	\end{proof}

	\begin{lemma}\label{lemma:convergenceformalpowers}
		If $\{f_n\}\subset C^{\infty}(\overline{J})$ is a sequence satisfying the conditions of Lemma \ref{Lemma:convergencefn}, then
		\begin{equation}\label{eq:convergenceformalpowers}
			\varphi_{f_n^{(-1)^j}}^{(k)}\rightarrow \varphi_{f^{(-1)^j}}^{(k)} \;\; \text{uniformly on } \overline{J} \;\; \text{for all } k\in \mathbb{N}_0,\; j=0,1.
		\end{equation}
	\end{lemma}
	\begin{proof}
		Since $f_n^{(-1)^j}\rightarrow f^{(-1)^j}$ uniformly on $\overline{J}$ for $j=0,1$, the proof proceeds identically to the argument given in \cite[Th. 10]{mineimpedance3}
	\end{proof}
	\begin{proposition}\label{prop:strongconvergence}
		If $\{f_n\}\subset C^{\infty}[0,\ell]$ is a sequence satisfying the conditions of Lemma \ref{Lemma:convergencefn}, then $\mathbf{T}_{f_n}\rightarrow \mathbf{T}_f$ in the strong operator topology of $\mathcal{B}(L^2(-\ell, \ell),L^2(0,\ell))$. The same conclusion holds for the strong operator topology of $\mathcal{B}(C[-\ell,\ell],C[0,\ell])$.
	\end{proposition}
	\begin{proof}
		Let $K_{f_n}(x,t)$ be the transmutation kernel associated with $f_n$. From the proof of Theorem \ref{Th:integralrepresentation1}, we obtain the estimate \eqref{eq:estimateintegralK}. Using the definition of $K_{f_n}$ \eqref{eq:kernelexponential}, we have
		\[
		\|K_{f_n}\|_{L^2(\mathcal{T}_{\ell})}^2 \leq 4\ell \left(d_n+2d_n^2(\ell^2d_n^2+\ell)e^{\ell d_n}\right),
		\]
		where $d_n=\left\|\frac{f_n'}{f_n}\right\|_{L^2(0,\ell)}^2$. Since $\frac{f_n'}{f}\rightarrow \frac{f'}{f}$ in $L^2(0,\ell)$ (by definition), we may set $D=\sup_{n\in \mathbb{N}}d_n$ and thus,
		\begin{equation}\label{eq:estimatenormskf}
			\sup_{n\in \mathbb{N}}\|K_{f_n}\|_{L^2(\mathcal{T}_{\ell})}^2\leq 4\ell \left(D+2D^2(\ell^2D^2+\ell)e^{\ell D}\right).
		\end{equation}
		It follows that the operators $\{\mathbf{T}_{f_n}\}$ are uniformly bounded in $\mathcal{B}(L^2(-\ell,\ell), L^2(0,\ell))$. On the other hand, for each $k\in \mathbb{N}_0$, Lemma \ref{lemma:convergenceformalpowers} gives
		\[
		\mathbf{T}_{f_n}[x^k]=\varphi_{f_n}^{(k)}\rightarrow \varphi_{f}^{(k)} \quad \text{in } L^2(0,\ell).
		\]
		By linearity, we conclude that $\mathbf{T}_{f_n} \rightarrow \mathbf{T}_f$ pointwise in the dense subspace $\mathcal{P}[-\ell,\ell]$ of $L^2(-\ell, \ell)$. According to \cite[Th. 4.9-6]{kreyszig}, uniformly boundedness of $\{\mathbf{T}_{f_n}\}$ yields that $\mathbf{T}_{f_n}\rightarrow \mathbf{T}_f$ in the strong operator topology of $\mathcal{B}(L^2(-\ell, \ell),L^2(0,\ell))$. 
		
		For the case $\mathcal{B}(C[-\ell,\ell], C[0,\ell])$, we use the estimate $\int_{-x}^{x}|K_{f_n}(x,t)|dt\leq 4( e^{\sqrt{\ell}d_n}-1)$. Hence $\sup_{n\in \mathbb{N}}\|\mathbf{T}_{f_n}\|_{\mathcal{B}(C[-\ell,\ell])}\leq 4e^{\sqrt{\ell}D}$, and the result follows from Lemma \ref{lemma:convergenceformalpowers}.
	\end{proof}
	
	\begin{lemma}
		The set of functions $\mathscr{E}=\{\phi(x)p(t)\chi_{(-x,x)}(t)\, |\, \phi\in C_0^{\infty}(0,\ell), p\in \mathcal{P}[-\ell,\ell]\}$ 
		is complete in $L^2(\mathcal{T}_{\ell})$.
	\end{lemma}
	\begin{proof}
		Let $g\in L^2(\mathcal{T}_{\ell})$ and assume $\iint_{\mathcal{T}_{\ell}}g\phi p\chi_{(-x,x)} =0$ for all $\phi p\chi_{(-x,x)} \in \mathscr{E}$. Fix $p$ and define $G[p](x)=\int_{-x}^{x}g(x,t)p(t)dt$. By Fubini's theorem
		\[
		0=\iint_{\mathcal{T}_{\ell}}g(x,t)\phi(x)p(t)\chi_{(-x,x)}(t)dxdt=\int_0^{\ell}\phi(x)G[p](x)dx.
		\]
		Since this equality holds for all $\phi\in C_0^{\infty}(0,\ell)$, we have that $G[p](x)=0$. Now, for $x$ fixed, $0=\int_{-x}^xg(x,t)p(t)dt$ for arbitrary $p \in \mathcal{P}[-\ell,\ell]$. In particular, $0=\int_{-x}^{x}g(x,t)p(t)dt$ for all $p \in \mathcal{P}[-x,x]$, and we conclude that for a.e. $x\in (0,\ell)$ and for a.e. $t\in (-x,x)$, $g(x,t)=0$. Then $g\equiv 0$ in $L^2(\mathcal{T}_{\ell})$ and $\mathscr{E}$ is complete.
	\end{proof}
	
	\begin{theorem}
		If $\{f_n\}\subset C^{\infty}[0,\ell]$ is a sequence satisfying the conditions of Lemma \ref{Lemma:convergencefn}, then $K_{f_n} \rightharpoonup K_f$ weakly in $L^2(\mathcal{T}_{\ell})$.
	\end{theorem}
	\begin{proof}
		Let $\Phi=\phi p\chi_{(-x,x)}\in \mathscr{E}$. Then
		\begin{align*}
			\langle \widehat{K}_{f_n}|\Phi\rangle_{L^2(\mathcal{T}_{\ell})}=\int_0^{\ell}\phi(x)\int_{-x}^{x}K_{f_n}(x,t)p(t)dt \rightarrow \int_0^{\ell}\phi(x)\int_{-x}^{x}K_{f}(x,t)p(t)dt=\langle \widehat{K}_f|\Phi\rangle_{L^2(\mathcal{T}_{\ell})},
		\end{align*}
		because $\int_{-x}^{x}K_{f_n}(x,t)p(t)dt$ converges uniformly to $\int_{-x}^{x}K_f(x,t)p(t)dt$ on $[0,\ell]$ (by Proposition \ref{prop:strongconvergence}). Since the sequence $\{K_{f_n}\}$ is bounded in $L^2(\mathcal{T}_{\ell})$ (by \eqref{eq:estimatenormskf}) and $\mathscr{E}$ is complete in $L^2(\mathcal{T}_{\ell})$, \cite[Lemma 4.8-7]{kreyszig} implies that $K_{f_n} \rightharpoonup K_f$ weakly in $L^2(\mathcal{T}_{\ell})$.
	\end{proof}
	
	\section{Neumann series of Bessel functions representation}
	
	Following \cite{NSBF1}, we develop the main properties of a Fourier-Legendre series representation for the transmutation kernel $K_f$. 
	
	Let $\{P_m\}_{m=0}^{\infty}$ be the system of orthogonal Legendre polynomials in $L^2(-1,1)$ with norm $\|P_m\|_{L^2(-1,1)}=\sqrt{\frac{2}{2m+1}}$. Fix $x\in (0,\ell]$. Hence $\{P_m\left(\frac{t}{x}\right)\}_{m=0}^{\infty}$ forms an orthogonal basis for $L^2(-x,x)$, and since $K_f(x,\cdot)\in L^2(-x,x)$, we have the Fourier-Legendre series expansion
	\begin{equation}\label{eq:FLseries}
		K_f(x,t)=\sum_{m=0}^{\infty}\frac{a_m(x)}{x}P_m\left(\frac{t}{x}\right),
	\end{equation}
	where the Fourier-Legendre coefficients are given by
	\begin{equation}\label{eq:fourierlegendrecoef1}
		a_m(x)=\left(m+\frac{1}{2}\right)\int_{-x}^{x}K_f(x,t)P_m\left(\frac{t}{x}\right)dt\qquad \forall m\in \mathbb{N}_0. 
	\end{equation}
	
	\begin{remark}\label{remark:fourierlegendrecoef}
		If we write $P_m(z)=\sum_{k=0}^{m}l_{k,m}z^k$, then using the mapping property \eqref{eq:mappingprop}, we get
		\begin{align*}
			a_m(x)&=\left(m+\frac{1}{2}\right)\int_{-x}^{x}K_f(x,t)\sum_{k=0}^{m}l_{k,m}\left(\frac{t}{x}\right)^kdt\\
			&= \left(m+\frac{1}{2}\right)\sum_{k=0}^{m}\frac{l_{k,m}}{x^k}(\mathbf{T}_f[x^k]-x^k) \\
			&= \left(m+\frac{1}{2}\right)\left(\sum_{k=0}^{m}\frac{l_{k,m}}{x^k}\varphi_f^{(k)}-P_m(1)\right).
		\end{align*}
		Since $P_m(1)=1$, this yields
		\begin{equation}\label{eq:fourierlegendrecoef}
			a_m(x)=\left(m+\frac{1}{2}\right)\left(\sum_{k=0}^{m}\frac{l_{k,m}}{x^k}\varphi_f^{(k)}(x)-1\right)\qquad \forall m\in \mathbb{N}_0.
		\end{equation}
		In particular, using $P_0\equiv 1$ and $P_1(z)=z$, we obtain the formulas
		\begin{equation}\label{eq:coefficients0and1}
			a_0(x)=\frac{f(x)-1}{2}\quad \text{and } a_1(x)=\frac{3}{2}\left(\frac{f(x)}{x}\int_0^x\frac{dt}{f^2(t)}-1\right)
		\end{equation}
		(because $\varphi_f^{(1)}(x)=f(x)\int_0^x\frac{dt}{f^2(t)})$. It is worth mentioning that relations of the type \eqref{eq:coefficients0and1} have proven to be a useful tool in the study of methods for solving inverse problems for regular potentials, where we seek to recover coefficients $a_0$ and $a_1$, and from them the solution $f$ \cite{vkinverse1,vkinverse2,vkstinverse2}.
	\end{remark}
	
	\begin{remark}
		Using the parity of the Legendre polynomials and the definition of the even and odd kernels $K_f^{\pm}(x,t)$, we deduce the following Legendre series:
		\begin{equation}\label{eq:FLseriessinecosine}
			K_f^{+}(x,t)=2\sum_{m=0}^{\infty}\frac{a_{2m}(x)}{x}P_{2m}\left(\frac{t}{x}\right) \quad \text{and}\;\;\; K_f^{-}(x,t)=2 \sum_{m=0}^{\infty}\frac{a_{2m+1}(x)}{x}P_{2m+1}\left(\frac{t}{x}\right)
		\end{equation}
	\end{remark}	
	
	Similar to the regular case \cite[Sec. 4]{NSBF1}, substitution of series \eqref{eq:FLseriessinecosine} in the integral relations \eqref{eq:transmutationcosinesine} leads to the following Neumann series of Bessel functions (NSBF) representations.
	
	\begin{theorem}
		For every $x\in (0,\ell]$, the solutions $C_f(\rho,x)$ and $S_f(\rho,x)$ admit the following series representations:
		\begin{align}
			C_f(\rho,x)& =\cos(\rho x)+\sum_{m=0}^{\infty}(-1)^m\alpha_{2m}(x)j_{2m}(\rho x), \label{eq:NSBFcosine}\\
			S_f(\rho,x) &= \frac{\sin(\rho x)}{\rho}+\frac{1}{\rho}\sum_{m=0}^{\infty}(-1)^m\alpha_{2m+1}(x)j_{2m+1}(\rho x),\label{eq:NSBFsine}
		\end{align}
		where $\alpha_m=2a_m$ and $j_{\nu}(\zeta)=\sqrt{\frac{\pi}{2\zeta}}J_{\nu+\frac{1}{2}}(\zeta)$ stands for the spherical Bessel functions.  For $N\in \mathbb{N}$, the partial sums
		\begin{align}
			C_f(\rho,x)& =\cos(\rho x)+\sum_{m=0}^{ \left[\frac{N}{2}\right] }(-1)^m\alpha_{2m}(x)j_{2m}(\rho x),\label{eq:NSBFpartialcosine} \\
			S_f(\rho,x) &= \frac{\sin(\rho x)}{\rho}+\frac{1}{\rho}\sum_{m=0}^{\left[\frac{N-1}{2}\right]}(-1)^m\alpha_{2m+1}(x)j_{2m+1}(\rho x),\label{eq:NSBFpartialsine}
		\end{align}
		obey the following estimates for every $\rho \in \mathbb{C}$ which lies in a strip of the form $|\operatorname{Im}\rho|\leq C$ with $C>0$:
		\begin{equation}\label{eq:estimatespartialsums}
			|C_f(\rho,x)-C_{f,N}(\rho,x)|\leq \epsilon_N(x) \frac{2\sinh(Cx)}{C} \quad \text{and}\;\;\; |\rho S_f(\rho,x)-\rho S_{f,N}(\rho,x)|\leq \epsilon_N(x) \frac{2\sinh(Cx)}{C},
		\end{equation}
		where $\epsilon_N(x)=\|K_f(x,\cdot)-K_{f,N}(x,\cdot)\|_{L^2(-x,x)}$. Then the series \eqref{eq:NSBFcosine} and \eqref{eq:NSBFsine} converge pointwise with respect to $x$ and uniformly with respect to $\rho$ in every strip $|\operatorname{Im}\rho|\leq C$.
	\end{theorem}	
	Even though formulas \eqref{eq:fourierlegendrecoef} provide an explicit way to compute the Fourier--Legendre coefficients, such expressions are typically impractical for numerical computation. Nevertheless, it is possible to compute these coefficients recursively, by extending the procedure developed in \cite[Sec. 6]{NSBF1} for the case of regular potentials.
	\begin{theorem}
		Define $\sigma_m(x)=x^m\alpha_m(x)$ for every $m\in \mathbb{N}_0$. The functions $\{\sigma_m(x)\}_{m=0}^{\infty}$ satisfies the following recursive relations:
		\begin{align}
			\sigma_0(x)& = f(x)-1, \quad \sigma_1(x)=3\left(f(x)\int_0^x\frac{dt}{f^2(t)}-x\right),\label{eq:sigmaoy1} \\
			\eta_m(x) &= \int_0^x\left(tf'(t)+(m-1)f(t)\right)\sigma_{m-2}(t)dt,\\
			\theta_m(x) &= \int_0^x\frac{1}{f^2(t)}\left(\eta_m(t)-tf(t)\sigma_{m-2}(t)\right)dt,\\
			\sigma_m(x) &= \frac{2m+1}{2n-3}\left[x^2\sigma_{m-2}(x)+c_mf(x)\theta_m(x)\right], \label{eq:sigmam}
		\end{align}  
		where $c_m=1$ if $m=1$ and $c_m=2(2m-1)$ otherwise.
	\end{theorem}	
	
	\begin{proof}
		Formulas \eqref{eq:sigmaoy1} follow from \eqref{eq:coefficients0and1}. Let $\{f_n\}\subset C^{\infty}[0,\ell]$ be a sequence of normalized non-vanishing functions converging to $f$ in the sense of Lemma \ref{Lemma:convergencefn}. Let $\{\sigma_{n,m}(x)\}$ denote the corresponding sequence associated with $f_n$. By formula \eqref{eq:fourierlegendrecoef},
		\begin{align*}
			\sigma_{n,m}(x)=\left(m+\frac{1}{2}\right)\left(\sum_{k=0}^{m}l_{k,m}x^{m-k}\varphi_{f_n}^{(k)}(x)-x^m\right)
		\end{align*} 
		Lemma \ref{lemma:convergenceformalpowers} implies that, for each fixed $m\in \mathbb{N}_0$, the right-hand side converges uniformly to $\sigma_m(x)$ on $[0,\ell]$ as $n\rightarrow \infty$.
		
		According to \cite[Sec. 6]{NSBF1}, for $f_n\in C^{\infty}[0,\ell]$, the coefficients $\{\sigma_{n,m}(x)\}$ satisfy the recursive relations \eqref{eq:sigmaoy1}-\eqref{eq:sigmam}. Since $\sigma_{n,m},f_n\rightarrow \sigma_m,f$ uniformly on $[0,\ell]$ and $f'_n\rightarrow f'$ in $L^2(0,\ell)$, it follows that $\eta_{n,m},\theta_{n,m}\rightarrow \eta_m,\theta_m$ uniformly on $[0,\ell]$ as $n\rightarrow \infty$. Hence the sequence $\{\sigma_m\}$ satisfies  \eqref{eq:sigmaoy1}-\eqref{eq:sigmam}.
	\end{proof}	
	
	\begin{remark}
		Let us denote by $\mathcal{R}:L^2(0,\ell)\rightarrow L^2(0,\ell)$ the reflection operator $\mathcal{R}u(x)=u(\ell-x)$. Hence $\mathcal{R}$ is a unitary operator in $L^2(0,\ell)$, and we extend its action to $q\in W^{-1,2}(0,\ell)$ by $(\mathcal{R}q|\phi)_{W_0^{1,2}(J)}:=(q|\mathcal{R}\phi)_{W_0^{1,2}(J)}$. A direct computation shows that $\mathcal{R}q=-(\mathcal{R}\sigma)'$. Hence, if $y\in W^{1,2}(J)$ is a solution of \eqref{eq:Schrodinger1}, then 
		\begin{align*}
			(\mathbf{S}_{\mathcal{R}q}\mathcal{R}y|\phi)_{W_0^{1,2}(J)}&=((\mathcal{R}y)'|\phi')_{W^{1,2}_0(J)}+(\mathcal{R}q|(\mathcal{R}y)\phi)_{W^{1,2}_0(J)}\\
			&=-(\mathcal{R}(y')|\phi')_{W^{1,2}_0(J)}+(q|\mathcal{R}((\mathcal{R}y)\phi))_{W^{1,2}_0(J)}\\
			&=(y'|(\mathcal{R}\phi)')_{W^{1,2}_0(J)}+(q|y(\mathcal{R}\phi))_{W^{1,2}_0(J)}\\
			&=\lambda (y|\mathcal{R}\phi)_{W^{1,2}_0(J)}=\lambda (\mathcal{R}y|\phi)_{W^{1,2}_0(J)},
		\end{align*}
		that is, $\mathcal{R}y$ is a solution of $\mathbf{S}_{\mathcal{R}q}\mathcal{R}y=\lambda \mathcal{R}y$. The initial conditions are expressed in terms of $-\mathcal{R}\sigma$. A non-vanishing solution is given by $\tilde{f}=\frac{\mathcal{R}f}{f(\ell)}$.
		
		Now, let $\psi(\rho,x)$ and $\vartheta(\rho,x)$ be the solutions of \eqref{eq:Schrodinger1} determined by the initial conditions at $x=\ell$:
		\begin{align}
			\psi(\rho,\ell)=1, & \quad \psi_{\sigma}^{[1]}(\rho,\ell)=0, \label{eq:initialpsi}\\
			\vartheta(\rho,\ell)=0, & \quad \vartheta_{\sigma}^{[1]}(\rho,\ell)=1. \label{eq:initialtheta}
		\end{align}
		Hence $\tilde{\psi}(\rho,x)=\mathcal{R}\psi(\rho,x)$, is a solution of $\mathbf{S}_{\mathcal{R}q}v=\lambda v$ satisfying the initial conditions $\tilde{\psi}(\rho,0)=1$ and 
		\[
		\tilde{\psi}_{-\mathcal{R}\sigma}^{[1]}(\rho,0)=-\mathcal{R}\psi'(\rho,0)+\mathcal{R}\sigma(0)\mathcal{R}\psi(\rho,0)=-\psi_{\sigma}^{[1]}(\rho,\ell)=0.
		\]
		Therefore, $\tilde{\psi}(\rho,x)=C_{\tilde{f}}(\rho,x)$. By \eqref{eq:NSBFcosine}, we conclude that $\psi(\rho,x)$ admits the NSBF representation 
		\begin{equation}\label{eq:NSBFpsi}
			\psi(\rho,x)=\cos(\rho(\ell-x))+\sum_{m=0}^{\infty}(-1)^mt_m(x)j_{2m}(\rho(\ell-x)).
		\end{equation}
		Analogously,
		\begin{equation}\label{eq:NSBFtheta}
			\vartheta(\rho,x)=\frac{\sin(\rho(\ell-x))}{\rho}+\sum_{m=0}^{\infty}(-1)^ms_m(x)j_{2m+1}(\rho(\ell-x)).
		\end{equation}
	\end{remark}
	\begin{remark}
		From relations \eqref{eq:Darbouxtransformsinecosine}, one can obtain NSBF representations for $\mathbf{D}_fC_f(\rho,x)$ and $\mathbf{D}_fS_f(\rho,x)$ in terms of the corresponding series of $C_{\frac{1}{f}}(\rho,x)$ and $S_{\frac{1}{f}}(\rho,x)$, and by Theorem \ref{th:factorizacion} (i), the NSBF representation for the $\sigma$-quasiderivatives.
	\end{remark}
	
	\section{Relation with a transmutation operator for the Sturm-Liouville equation in impedance form }
	
	In this final section, we analyze an impedance equation of the form 
	\begin{equation}\label{eq:impedance1}
		\mathbf{L}_fv=-\frac{1}{f^2}Df^2Dv=\lambda v,\quad v\in W^{2,2}(0,\ell),
	\end{equation}
	where $f\in W^{1,2}(0,\ell)$ is a complex-valued nonvanishing function. The function $f$ is called the {\it impedance function} of Eq. \eqref{eq:impedance1}. All concepts associated with the operator $\mathbf{L}_f$ (solutions, formal powers, etc) will be denoted with a hat. Let $\widehat{e}_f(\rho,x)$ be the solution of \eqref{eq:impedance1} satisfying the initial conditions
	\begin{equation}\label{eq:impedanceosolexp}
		\widehat{e}_f(\rho,0)=1,\qquad (\widehat{e}_f)'(\rho,0)=i\rho.
	\end{equation} 
	
	The formal powers associated with the impedance $f$ are the functions $\{\widehat{\varphi}_f^{(k)}\}_{k=0}^{\infty}$ defined by:
	\begin{align*}
		\widehat{\varphi}_f^{(0)}\equiv 1, & \quad \widehat{\varphi}_f^{(1)}(x)=\int_{0}^{x}\frac{dt}{f^2(t)}, \\
		\widehat{\varphi}_f^{(k)}(x) & = k(k-1)\int_0^x\frac{dt}{f^2(t)}\int_0^tf^2(s)\widehat{\varphi}_f^{(k-2)}(s)ds,\quad k\geq 2.
	\end{align*}
	
	In \cite[Th. 5]{mineimpedance3}, the SPPS representation for the solution $\widehat{e}_f(\rho,x)$ is established:
	\[
	\widehat{e}_f(\rho,x)=\sum_{k=0}^{\infty}\frac{(i\rho)^k\widehat{\varphi}_f^{(k)}(x)}{k!}.
	\]
	The series converges with respect to the variable $x$ in the norm of $W^{2,2}(0,\ell)$, and uniformly on compact subsets of the $\rho$-complex plane.
	
	According to \cite{mineimpedance3}, the impedance equation \eqref{eq:impedance1} admits a transmutation kernel. 
	\begin{theorem}[\cite{mineimpedance3}]\label{th:impedance1}
		There exists a kernel $\widehat{K}_f\in L^2(\mathcal{T}_{\ell})$ such that
		
		\begin{equation}
			\widehat{e}_f(\rho,x)=e^{i\rho x}-i\rho \int_{-x}^{x}\widehat{K}_f(x,t)e^{i\rho t}dt.
		\end{equation}
		Furthermore, if we consider the operator
		\begin{equation}
			\widehat{\mathbf{T}}_fv(x)=v(x)-\int_{-x}^{x}\widehat{K}_f(x,t)v'(t)dt,\quad v\in W^{1,2}(-\ell,\ell),
		\end{equation}
		then the action of $\widehat{\mathbf{T}}_f$ over the powers $\{x^k\}_{k=0}^{\infty}$ is given by
		\begin{equation}\label{eq:transmpropimpedance}
			\widehat{\mathbf{T}}_f[x^k]=\widehat{\varphi}_f^{(k)}(x)\qquad \forall k\in \mathbb{N}_0,
		\end{equation}
		and the following transmutation property holds:
		\begin{equation}
			\mathbf{L}_f\widehat{\mathbf{T}}_fv=-\widehat{\mathbf{T}}_fv''\qquad \forall v\in W^{3,2}(-\ell,\ell).
		\end{equation}
	\end{theorem}
	The proof of these facts can be found in \cite[Sec. 6]{mineimpedance3}.
	
	Consider the operator $\mathbf{R}_f: L^2(0,\ell; |f|^2dx)\rightarrow L^2(0,\ell)$ given by $\mathbf{R}_fv=fv$. Note that $\mathbf{R}_f^{-1}=\mathbf{R}_{\frac{1}{f}}$. Hence the Liouville transform relating $\mathbf{S}_f$ with $\mathbf{L}_f$ is given by $\mathbf{R}_{\frac{1}{f}}: \mathscr{D}_2(\mathbf{S}_f)\rightarrow W^{2,2}(0,\ell)$, and
	\[
	\mathbf{S}_f y= \mathbf{R}_{f}\mathbf{L}_f\mathbf{R}_{\frac{1}{f}}y \qquad \forall y\in \mathscr{D}_2(\mathbf{S}_f).
	\]
	According to Remark \ref{Remark:initialconditionsunitaryoperator}, if we set $v=\mathbf{R}_{\frac{1}{f}}e_f(\rho,x)$, then $v$ is the unique solution of \eqref{eq:impedance1} satisfying the initial conditions $v(0)=1$ and $v'(0)=\mathbf{D}_fe_f(\rho,0)=i\rho$. Thus, $\widehat{e}_f(\rho,x)=\mathbf{R}_{\frac{1}{f}}e_f(\rho,x)$. Comparing the SPPS expansions of $e_f(\rho,x)$ and $\widehat{e}_f(\rho,x)$ we obtain that 
	\begin{equation}\label{eq:liouvilleformalpowers}
		\widehat{\varphi}_f^{(k)}(x)=\mathbf{R}_{\frac{1}{f}}\varphi_f^{(k)}(x)\qquad \forall k\in \mathbb{N}_0.
	\end{equation}
	
	 This relation between the formal powers yields the corresponding relation between the transmutation operators $\mathbf{T}_f$ and $\widehat{T}_f$.
	
	\begin{theorem}\label{Th:Liouvilletransmutationrelation}
		The following statements hold:
		\begin{itemize}
			\item[(i)] The operators $\mathbf{T}_f$ and $\widehat{\mathbf{T}}_f$ are related as follows:
			\begin{equation}\label{eq:transmliouville}
				\widehat{\mathbf{T}}_fv= \mathbf{R}_{\frac{1}{f}}\mathbf{T}_fv \qquad \forall v\in W^{1,2}(-\ell,\ell).
			\end{equation}
			\item[(ii)] For every $x\in (0,\ell]$, $\widehat{K}_f(x,\cdot)\in W^{1,2}(-x,x)$ and 
			\begin{equation}\label{eq:kernelderivative}
				\frac{\partial \widehat{K}_f(x,t)}{\partial t}=\frac{1}{f(x)}K_f(x,t) \qquad \text{a.e. in }\; \mathcal{T}_{\ell}.
			\end{equation}
			\item[(iii)] The kernel $\widehat{K}_f$ satisfies the Goursat conditions
			\begin{equation}\label{eq:goursatimpedance}
				\widehat{K}_f(x,x)=1-\frac{1}{f(x)},\qquad \widehat{K}_f(x,-x)=0.
			\end{equation}
		\end{itemize}
	\end{theorem}
	\begin{proof}
		\begin{itemize}
			\item[(i)] By linearity, relation \eqref{eq:liouvilleformalpowers} holds in $\mathcal{P}[-\ell,\ell]$. Given $v\in W^{1,2}(-\ell,\ell)$, chose a sequence $\{p_n\}\subset \mathcal{P}[-\ell,\ell]$ such that $p_n\rightarrow v$ in $W^{1,2}(-\ell,\ell)$. Since $p_n'\rightarrow v'$ in $L^2(-\ell,\ell)$ and $\widehat{K}_f\in L^2(\mathcal{T}_{\ell})$, hence $\widehat{T}_fp_n\rightarrow \widehat{T}_fv$ in $L^2(-\ell,\ell)$. On the other hand, $\frac{1}{f}\mathbf{T}_fp_n\rightarrow \frac{1}{f}\mathbf{T}_fv$ in $L^2(-\ell,\ell)$, because $\mathbf{T}_f\in \mathcal{B}(L^2(-\ell,\ell), L^2(0,\ell))$ and $\frac{1}{f}\in C[0,\ell]$. Hence, passing to the limit, we obtain \eqref{eq:transmliouville}.
			\item[(ii)] By the previous point, the following relation is valid for all $\phi\in C^{\infty}[-\ell,\ell]$:
			\begin{equation}\label{eq:auxiliar2}
				\phi(x)-\int_{-x}^{x}\widehat{K}_f(x,t)\phi'(t)dt=\frac{\phi(x)}{f(x)}+\int_{-x}^{x}\frac{K_f(x,t)}{f(x)}\phi(t)dt.
			\end{equation}
			By Theorem \ref{th:impedance1}, the left-hand side belongs to $W^{2,2}(-\ell,\ell)$ and, in particular, is a continuous function. The right-hand side is continuous by Theorem \ref{th:continuityofoperator}. Hence equality \eqref{eq:auxiliar2} holds for all $x\in [0,\ell]$. Fix $x\in (0,\ell]$, and take $\phi \in C^{\infty}_0(-x,x)$. Thus,
			\[
			\int_{-x}^{x}\widehat{K}_f(x,t)\phi'(t)dt= \int_{-x}^{x}\frac{K_f(x,t)}{f(x)}\phi(t)dt.
			\]
			This implies that $\frac{\partial K_f(x,t)}{\partial t}=\frac{K_f(x,t)}{f(x)}$. Since $\frac{K_f(x,\cdot)}{f(x)}\in L^2(-x,x)$, we conclude (ii), as desired.
			
			\item[(iii)] Integration by parts in the left hand side of \eqref{eq:auxiliar2} yields
			\begin{align*}
				(1-\widehat{K}_f(x,x))\phi(x)+\widehat{K}_f(x,-x)\phi(-x)+\int_{-x}^{x}\frac{\partial \widehat{K}_f(x,t)}{\partial t}\phi(t)dt=\frac{\phi(x)}{f(x)}+\int_{-x}^{x}\frac{K_f(x,t)}{f(x)}\phi(t)dt
			\end{align*}
			and by the point (ii)
			\[
			(1-\widehat{K}_f(x,x))\phi(x)+\widehat{K}_f(x,-x)\phi(-x)=\frac{\phi(x)}{f(x)} \qquad \forall \phi \in C^{\infty}[-\ell,\ell].
			\]
			This implies \eqref{eq:goursatimpedance}.
		\end{itemize}
	\end{proof}
	\begin{corollary}
		\begin{itemize}
			\item[(i)]  $\widehat{\mathbf{T}}_f$ admits a continuous extension in $ \mathcal{B}(C[-\ell,\ell],C[0,\ell])$.
			\item[(ii)] $\widehat{\mathbf{T}}_f\in \mathcal{B}(W^{1,2}(-\ell,\ell), W^{1,2}(0,\ell))$.
		\end{itemize}
		
	\end{corollary}
	\begin{proof}
		The statement (i) follows from Theorem \ref{Th:transmdarboux}(i), from the fact that the right-hand side of \eqref{eq:transmliouville} defines a bounded operator from $C[-\ell,\ell]$ to $C[0,\ell]$ (by Theorem \ref{th:continuityofoperator}), and from the density of $W^{1,2}(-\ell,\ell)$ in $C[-\ell,\ell]$.

		For (ii),  let $\phi \in C^{\infty}[-\ell,\ell]$. By Theorem \ref{Th:transmdarboux}, $fD\left(\frac{1}{f}\mathbf{T}_f\phi \right)=\mathbf{T}_{\frac{1}{f}}\phi'$, and by Theorem \ref{Th:Liouvilletransmutationrelation}(i) we obtain
		\begin{equation}\label{eq:relationderivativetftdarboux}
			D\widehat{\mathbf{T}}_f\phi = \frac{1}{f}\left[fD\left(\frac{1}{f}\mathbf{T}_f\phi \right)\right]=\frac{1}{f}\mathbf{T}_{\frac{1}{f}}\phi'.
		\end{equation}
		Hence
		\[
		\|D\widehat{\mathbf{T}}_f\phi\|_{L^2(-\ell,\ell)}\leq \|f^{-1}\|_{L^{\infty}(0,\ell)}\|\mathbf{T}_{\frac{1}{f}}\|_{\mathcal{B}(L^2(-\ell,\ell),L^2(0,\ell))}\|\phi'\|_{L^2(-\ell,\ell)},
		\]
		and consequently, $\|\widehat{\mathbf{T}}_f\|_{W^{1,2}(0,\ell)}\leq M \|\phi\|_{W^{1,2}(-\ell,\ell)}$, where\\ $M=\max\{1+\|\widehat{K}_f\|_{L^2(\mathcal{T}_{\ell})},\|f^{-1}\|_{L^{\infty}(0,\ell)}\|\mathbf{T}_{\frac{1}{f}}\|_{\mathcal{B}(L^2(-\ell,\ell),L^2(0,\ell))}\}$. The result follows from the fact that $C^{\infty}[-\ell,\ell]$ is dense in $W^{1,2}(-\ell,\ell)$.
	\end{proof}
	
	\begin{lemma}\label{lemma:convergencesobolev}
		Let $\Omega\subset \mathbb{R}^N$ be a bounded Lipschitz domain, and let $g\in L^2(\Omega)$. Suppose that there is a sequence $\{g_n\}\subset W^{1,2}(\Omega)$ such that:
		\begin{itemize}
			\item[(i)] $g_n  \rightharpoonup g$ weakly in $L^2(\Omega)$.
			\item[(ii)] $\displaystyle \sup_{n\in \mathbb{N}}\|\nabla g_n\|_{L^2(\Omega)}<\infty$. 
		\end{itemize}
		Then $g\in W^{1,2}(\Omega)$ and $g_n \rightharpoonup g$ weakly in $L^2(\Omega)$. Moreover, after passing to a subsequence, $g_n\rightarrow g$ in $L^2(\Omega)$ and $g_n\rightarrow g$ a.e. in $\Omega$.
	\end{lemma}
	\begin{proof}
		Let $\phi \in C^{\infty}_0(\Omega)$. By hypothesis, $\langle g|\phi_{x_i} \rangle_{L^2(\Omega)}=\lim_{n\rightarrow \infty} \langle g_n|\phi_{x_i}\rangle_{L^2(\Omega)}$, where $\phi_{x_i}=\frac{\partial \phi}{\partial x_i}, i=1,\dots, N$. Hence
		\[
		|\langle g|\phi_{x_i}\rangle_{L^2(\Omega)}|=\lim_{n\rightarrow\infty}|\langle g_n|\phi_{x_i}\rangle_{L^2(\Omega)}|=\lim_{n\rightarrow\infty}|\langle (g_n)_{x_i}|\phi \rangle_{L^2(\Omega)}|\leq M\|\phi\|_{L^2(\Omega)},
		\]
		where $M=\sup_{n\in \mathbb{N}}\|\nabla g_n\|_{L^2(\Omega)}$. According to \cite[Prop. 9. 3]{brezis}, this implies that $g\in W^{1,2}(\Omega)$. Now, for $\phi\in C_0^{\infty}(\Omega)$, 
		\[
		\langle g_{x_i}|\phi\rangle_{L^2(\Omega)}= \langle g|\phi_{x_i}\rangle_{L^2(\Omega)}=\lim_{n\rightarrow \infty}\langle g_n|\phi_{x_i}\rangle_{L^2(\Omega)}=\lim_{n\rightarrow \infty}\langle (g_n)_{x_i}|\phi\rangle_{L^2(\Omega)}.
		\]
		By the density of $C_0^{\infty}(\Omega)$ and condition (ii), we conclude that $(g_n)_{x_i}\rightharpoonup g_{x_i}$ weakly in $L^2(\Omega)$. Then condition (i) implies that  $g_n \rightharpoonup g$ weakly in $W^{1,2}(\Omega)$.
		
		Since $\Omega$ is a bounded Lipschitz domain, \cite[Th. 3.27 and Th. 3.30]{mclean} imply that the embedding $W^{1,2}(\Omega) \hookrightarrow L^2(\Omega)$ is compact. Since $\{g_n\}$ is bounded in $W^{1,2}(\Omega)$, we obtain a subsequence, which we denote by simplicity as $\{g_n\}$, such that $g_n \rightarrow \tilde{g}$ in $L^2(\Omega)$. By (i), we have that $\tilde{g}=g$. Finally, by \cite[Th. 4.9]{brezis} there exists a subsequence such that $g_n\rightarrow g$ a.e. in $\Omega$.
	\end{proof}
	
	\begin{lemma}\label{lemma:convergenceweakhilbert}
		Let $\mathcal{H}$ be a Hilbert space, $u\in \mathcal{H}$ and $\{u_n\}\subset \mathcal{H}$ such that $u_n \rightharpoonup u$ weakly in $\mathcal{H}$. If $\{v_n\}\subset \mathcal{H}$ is a sequence such that $v_n\rightarrow v$ in $\mathcal{H}$, then $\langle u|v\rangle_{\mathcal{H}}=\lim_{n\rightarrow \infty}\langle u_n|v_n\rangle_{\mathcal{H}}$.
	\end{lemma}
	\begin{proof}
		Since $u_n \rightharpoonup u$ weakly in $\mathcal{H}$, then $M=\sup_{n\in \mathbb{N}}\|u_n\|_{\mathcal{H}}<\infty$. Hence
		\begin{align*}
			|\langle u_n|v_n\rangle_{\mathcal{H}}-\langle u|v\rangle_{\mathcal{H}}|\leq |\langle u_n|v_n-v\rangle_{\mathcal{H}}|+ |\langle u_n|v\rangle_{\mathcal{H}}-\langle u|v\rangle_{\mathcal{H}}|\leq M \|v_n-v\|_{\mathcal{H}}+|\langle u_n|v\rangle_{\mathcal{H}}-\langle u|v\rangle_{\mathcal{H}}|,
		\end{align*}
		and the right-hand side tends to zero as $n\rightarrow \infty$, by the hypotheses.
	\end{proof}
	\begin{remark}
		By \eqref{eq:kernelderivative}, the partial derivative $\frac{\partial \widehat{K}_f}{\partial t}$ satisfies the estimate 
		\begin{equation}\label{eq:estimatederivative}
			\sup_{0< x\leq \ell} \int_{-x}^{x}\left|\frac{\partial \widehat{K}_f(x,t)}{\partial t}\right|^2 \leq 8\ell\|f^{-1}\|_{L^{\infty}(0,\ell)}\ell \left(d_f+2d_f^2(\ell^2d_f+\ell)e^{\ell d_f}\right), \quad \text{where } d_f=\left\|\frac{f'}{f}\right\|_{L^2(0,\ell)}.
		\end{equation}
	\end{remark}
    
	\begin{theorem}
		The transmutation kernel $\widehat{K}_f\in C(\overline{\mathcal{T}_{\ell}})\cap W^{1,2}(\mathcal{T}_{\ell})$.
		
		Furthermore, $\widehat{K}_f$ and $\widehat{K}_{\frac{1}{f}}$ satisfy the relations
		\begin{equation}\label{eq:relationskernelsimpedance}
			\frac{\partial \widehat{K}_{\frac{1}{f}}(x,t)}{\partial x}= -f^2(x)\frac{\partial \widehat{K}_f(x,t)}{\partial t}\quad\text{and}\quad \frac{\partial \widehat{K}_f(x,t)}{\partial x}=-\frac{1}{f^2(x)}\frac{\partial \widehat{K}_{\frac{1}{f}}(x,t)}{\partial t}.
		\end{equation}
	\end{theorem}
	
	\begin{proof}
		We know that $\frac{\partial \widehat{K}_f}{\partial t}$ exists and satisfies  \eqref{eq:estimatederivative}, which implies that $\frac{\partial \widehat{K}_f}{\partial t}\in L^2(\mathcal{T}_{\ell})$. Let  $\{f_n\}\subset C^{\infty}[0,\ell]$ be a sequence of non-vanishing normalized impedance such that $f_n\to f$ in the sense of Lemma \ref{Lemma:convergencefn}. Let $\{\widehat{K}_{f_n}\}$ be the corresponding sequence of transmutation kernels. We divide the proof into  steps:
		\newline 
		
		{\it Step I: $\widehat{K}_{f_n}\rightharpoonup \widehat{K}_f$ weakly in $L^2(\mathcal{T}_{\ell}$)}. According to \cite[Th: 10]{mineimpedance3}, $\widehat{\varphi}_{f_n}^{(k)}\rightarrow \widehat{\varphi}_f^{(k)}$ uniformly on $[0,\ell]$ for all $k\in \mathbb{N}_0$. Given $p\in \mathcal{P}[-\ell,\ell]$, by \eqref{eq:transmpropimpedance}, 
		\[
		\int_{-x}^{x}\widehat{K}_{f_n}(x,t)p(t)dt=\sum_{k=0}^{N}p_k\int_{-x}^{x}\widehat{K}_{f_n}(x,t)t^k=\sum_{k=0}^{N}p_k\frac{\widehat{\varphi}_{f_n}^{(k+1)}(x)-x^{k+1}}{k+1} 
		\]
		and the right-hand side converges uniformly to $\sum_{k=0}^{N}p_k\frac{\widehat{\varphi}_f^{(k+1)}(x)-x^{k+1}}{k+1}=\int_{-x}^{x}\widehat{K}_f(x,t)p(t)dt$. This implies that $\langle \widehat{K}_{f_n}|\Phi \rangle_{L^2(\mathcal{T}_{\ell})}\rightarrow \langle \widehat{K}_f|\Phi \rangle_{L^2(\mathcal{T}_{\ell})}$ for all $\Phi \in \mathscr{E}$.
		
		In \cite[Sec. 6]{mineimpedance3} and \cite[Sec. 3. 1. 2 ]{mitesis}, the following estimate was established::
		\[
		\|\widehat{K}_{f_n}\|_{L^2(\mathcal{T}_{\ell})}^2 \leq \frac{ c^2\ell }{\pi}Ie^{4\left\|\frac{f'_n}{f_n}\right\|_{L^1(0,\ell)}},	
		\]
		where $I= \int_0^{\infty}\frac{dt}{1+t^2}$ and $c=\max\left\{\max_{|z|<1}(1+|z|)e^{-|\operatorname{Im}z|},2\right\}$. Since $\frac{f_n'}{f_n}\rightarrow \frac{f'}{f}$ in $L^2(0,\ell)$, we have $D=\sup_{n\in \mathbb{N}}\left\|\frac{f_n'}{f_n}\right\|_{L^2(0,\ell)}<\infty$, and 
		\[
		\sup_{n\in \mathbb{N}} \|\widehat{K}_{f_n}\|_{L^2(\mathcal{T}_{\ell})}^2 \leq \frac{ c^2\ell }{\pi}Ie^{4\sqrt{\ell}D}
		\]
		The density of $\operatorname{Span}\mathscr{E}$ implies that $\widehat{K}_{f_n}\rightharpoonup \widehat{K}_f$ weakly in $L^2(\mathcal{T}_{\ell})$.
		\newline
		
		\emph{Step II: $\widehat{K}_{f_n}\rightharpoonup \widehat{K}_f$ weakly in $W^{1,2}(\mathcal{T}_{\ell})$ }.  By \eqref{eq:estimatederivative}, 
		\begin{equation}\label{eq:auxiliar3}
			\sup_{n\in \mathbb{N}}\left\|\frac{\partial \widehat{K}_{f_n}}{\partial t}\right\|_{L^2(\mathcal{T}_{\ell})}^2\leq 8\ell^2M\ell( D+2D^2(\ell^2D+\ell)e^{\ell D}),
		\end{equation}
		
		where $M=\sup_{n\in \mathbb{N}}\|f_n^{-1}\|_{L^{\infty}(0,\ell)}$ (which is finite because $f_n^{-1}$ converges to $f^{-1}$ uniformly on $[0,\ell]$). Since $f_n\in C^1[0,\ell]$,  it follows from \cite[Prop. 5]{mineimpedance1} that $\{\widehat{K}_{f_n},\widehat{K}_{\frac{1}{f_n}}\}\subset C^1(\overline{\mathcal{T}_{\ell}})$, and according to \cite[Th. 34]{mineimpedance3}, the following relation holds
		\begin{equation}\label{eq:auxiliar4}
			\frac{\partial \widehat{K}_{f_n}(x,t)}{\partial x}=-\frac{1}{f_n^2(x)}\frac{\partial \widehat{K}_{\frac{1}{f_n}}(x,t)}{\partial t}\qquad \text{ for all  } (x,t)\in \overline{\mathcal{T}_{\ell}}.
		\end{equation}
		Thus, by \eqref{eq:auxiliar3} we have that $\sup_{n\in \mathbb{N}}\left\|\frac{\partial \widehat{K}_{f_n}(x,t)}{\partial x}\right\|_{L^2(\mathcal{T}_{\ell})}^2< \infty$. By Lemma \ref{lemma:convergencesobolev}, we conclude that $\widehat{K}_{f_n}\rightharpoonup \widehat{K}_f$ in $W^{1,2}(\mathcal{T}_{\ell})$, and, in particular, that $\widehat{K}_f\in W^{1,2}(\mathcal{T}_{\ell})$.
		
		Since $\frac{\partial \widehat{K}_{f_n}}{\partial x}, \frac{\partial \widehat{K}_{f_n}}{\partial t}$ converge weakly to $\frac{\partial \widehat{K}_{f_n}}{\partial x}, \frac{\partial \widehat{K}_{f_n}}{\partial t}$  in $L^2(\mathcal{T}_{\ell}$), then relation \eqref{eq:auxiliar4} implies \eqref{eq:relationskernelsimpedance}.
		\newline
		
		\emph{Step III: $\widehat{K}_f\in C(\overline{\mathcal{T}_{\ell}})$.}  Let us consider the change of variables $\xi=\frac{x+t}{2}, \zeta=\frac{x-t}{2}$, which transform $\mathcal{T}_{\ell}$ into the triangle $\mathcal{Q}_{\ell}= \{ (\xi,\zeta)\in \mathbb{R}^2\, |\, \xi,\zeta\geq 0 \text{ and } \xi+\zeta< \ell\}$ (see Figure \ref{fig:triangulos}).
		
		\begin{figure}[h]
			\centering
			\includegraphics[width=11.5cm]{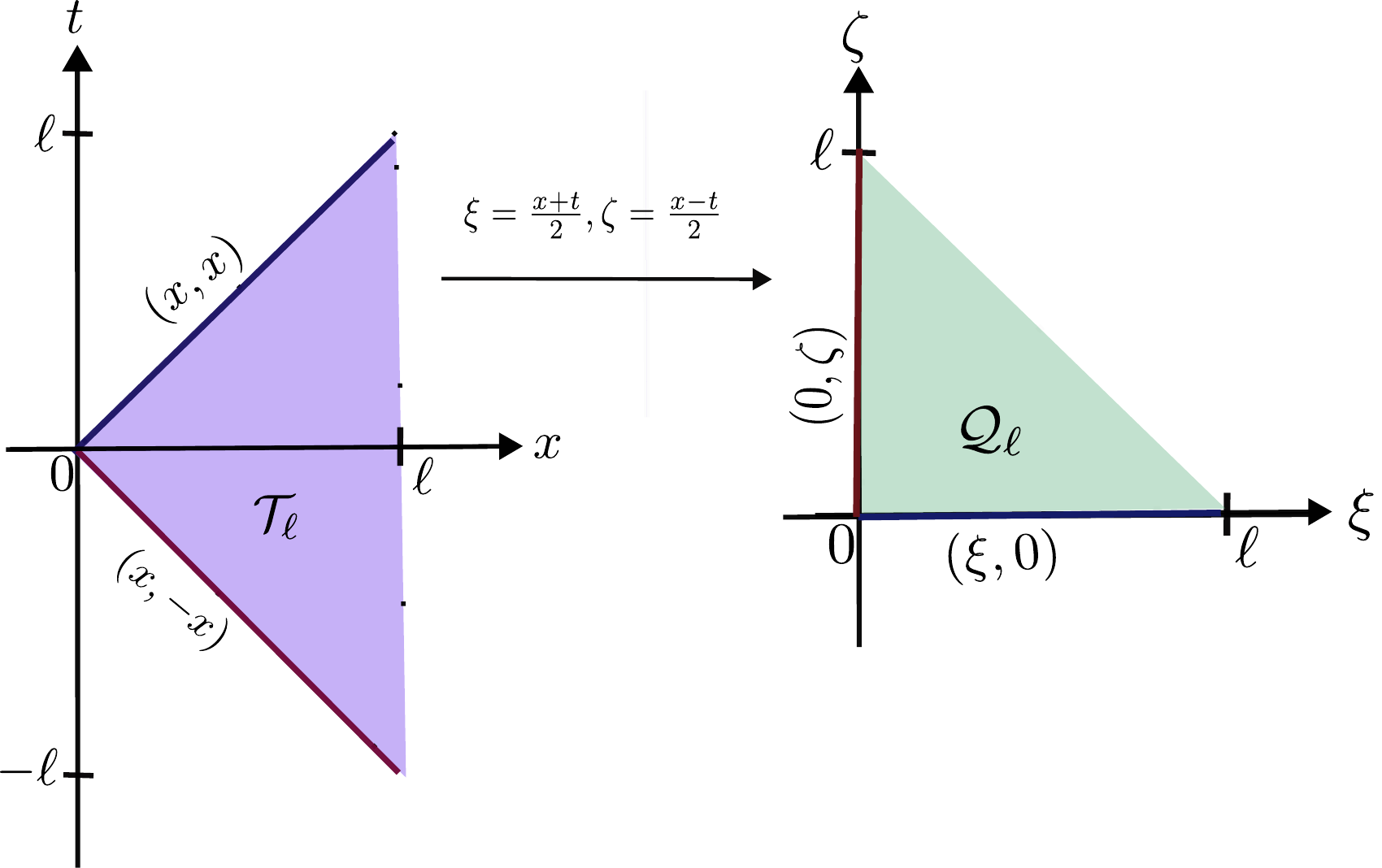}
			\caption{The domains $\mathcal{T}_{\ell}$ and $\mathcal{Q}_{\ell}$ and the transformation between them.}
			\label{fig:triangulos}
		\end{figure}
		
		Denote $\tau_n=-2\frac{f'_n}{f_n}$ and $\tau=-2\frac{f'}{f}$, so $\tau_n\rightarrow \tau$ in $L^2(0,\ell)$.

		Define $H_n(\xi,\zeta)=\widehat{K}_{f_n}(\xi+\zeta,\xi -\zeta)$ and $H(\xi,\zeta)=\widehat{K}_f(\xi+\zeta,\xi -\zeta)$. Since this change of variables is an invertible linear transformation, it follows that $H_n \rightharpoonup H$ weakly in $W^{1,2}(\mathcal{Q}_{\ell})$. By Lemma \ref{lemma:convergencesobolev}, passing to a subsequence, we may assume that $H_n\rightarrow H$ a.e. in $\mathcal{Q}_{\ell}$.

		According to \cite[Th. 4 and Prop. 5]{mineimpedance1}, the functions $H_n$ satisfy the  integro-differential equation
		\begin{equation}\label{eq:integraleq1}
			H_n(\xi,\zeta)=1-\frac{1}{f_n(\xi)}+\frac{1}{2}\int_0^{\xi}d\alpha\int_0^{\zeta}d\beta\tau_n(\alpha+\beta)\mathfrak{D}H_n(\alpha,\beta)\qquad \forall (\xi,\zeta)\in \overline{\mathcal{Q}_{\ell}},
		\end{equation}
		where $\mathfrak{D}:=\frac{\partial}{\partial x}+\frac{\partial}{\partial t}$. Observe that $\mathfrak{D}H_n(\xi,\zeta)=2\frac{\partial \widehat{K}_n}{\partial x} (\xi+\zeta,\xi-\zeta)$, so $\mathfrak{D}H_n\rightharpoonup \mathfrak{D}H$ weakly in $L^2(\mathcal{Q}_{\ell})$.

		For $(\xi,\zeta)\in \mathcal{Q}_{\ell}$, let $C_{(\xi,\zeta)}$ denote the rectangle in $\mathcal{Q}_{\ell}$ with vertices $(0,0), (0,\xi), (\xi,\zeta)$ and $(\zeta,0)$ (see Figure \ref{fig:square}). Set $\tilde{\tau}_n^{(\xi,\zeta)}(\alpha,\beta)=\tau_n(\alpha+\beta)\chi_{C_{(\xi,\zeta)}}(\alpha,\beta)$. It is a straightforward to show that for every $(\xi,\zeta)\in \mathcal{Q}_{\ell}$, $\tilde{\tau}_n^{(\xi,\zeta)}\rightarrow \tilde{\tau}^{(\xi,\zeta)}$ in $L^2(\mathcal{Q}_{\ell})$.
		\begin{figure}[h]
			\centering
			\includegraphics[width=5.5cm]{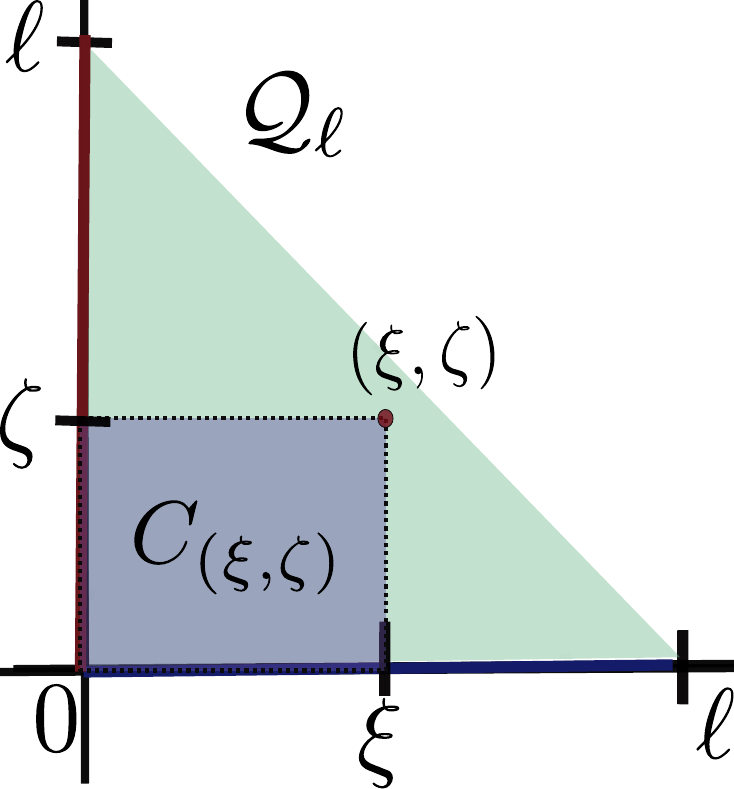}
			\caption{The rectangle $C_{(\xi,\zeta)}$ contained within $\mathcal{Q}_{\ell}$.}
			\label{fig:square}
		\end{figure}
		With these notations, \eqref{eq:integraleq1} can be rewritten as
		\begin{equation*}
			H_n(\xi,\zeta)=1-\frac{1}{f_n(\xi)}+\frac{1}{2}\iint_{\mathcal{T}_{\ell}} \tilde{\tau}_n^{(\xi,\zeta)}\mathfrak{D}H_n =1-\frac{1}{f_n(\xi)}+\frac{1}{2}\langle \tilde{\tau}_n^{(\xi,\zeta)}|\overline{\mathfrak{D}H_n} \rangle_{L^2(\mathcal{Q}_{\ell})}
		\end{equation*}
		The left-hand side converges a.e. to $H$, while $1-\frac{1}{f_n}\rightarrow 1-\frac{1}{f}$ uniformly on $\xi$. Since $\mathfrak{D}H_n\rightharpoonup \mathfrak{D}H$ weakly in $L^2(\mathcal{Q}_{\ell})$ and $\tilde{\tau}_n^{(\xi,\zeta)}\rightarrow \tilde{\tau}^{(\xi,\zeta)}$ strongly, Lemma \ref{lemma:convergenceweakhilbert} yields
		\begin{align*}
			H(\xi,\zeta)& =1-\frac{1}{f(\xi)}+\frac{1}{2}\langle \tilde{\tau}^{(\xi,\zeta)}|\overline{\mathfrak{D}H} \rangle_{L^2(\mathcal{Q}_{\ell})}\\
			&= 1-\frac{1}{f(\xi)}+\frac{1}{2}\int_0^{\xi}d\alpha\int_0^{\zeta}d\beta\tau(\alpha+\beta)\mathfrak{D}H(\alpha,\beta).
		\end{align*}
		The expression on the right-hand side is continuous in $\overline{\mathcal{Q}_{\ell}}$, so $H\in C(\overline{\mathcal{Q}_{\ell}})$. Therefore, $\widehat{K}_f\in C(\overline{\mathcal{T}_{\ell}})$.
	\end{proof}

	\begin{remark}\label{eq:remarkGoursatproblem}
		According to \cite[Remark 9]{mineimpedance1}, each kernel $\widehat{K}_{f_n}$ satisfies 
		\[
		\iint_{\mathcal{T}_{\ell}}f_n^2\left\{(\widehat{K}_{f_n})_x\phi_x-(\widehat{K}_{f_n})_t\phi_t\right\}=0\qquad \forall \phi \in C^{\infty}_0(\mathcal{T}_{\ell}).
		\]
		This can be rewritten as 
		\[
		\langle (\widehat{K}_{f_n})_x|\overline{f_n^2\phi_x}\rangle_{L^2(\mathcal{T}_{\ell})}-\langle (\widehat{K}_{f_n})_t|\overline{f_n^2\phi_t}\rangle_{L^2(\mathcal{T}_{\ell})}=0.
		\]
		Since $f_n^2\rightarrow f^2$ uniformly on $[0,\ell]$, it follows that $f_n^2\phi_x,f_n^2\phi_t\rightarrow f^2\phi_x,f^2\phi_t$ strongly in $L^2(\mathcal{T}_{\ell})$. Since $(\widehat{K}_{f_n})_x,(\widehat{K}_{f_n})_t\rightharpoonup (\widehat{K}_f)_x,(\widehat{K}_f)_t$ weakly in $L^2(\mathcal{T}_{\ell})$,  Lemma \ref{lemma:convergenceweakhilbert} implies 
		\begin{equation}\label{eq:hypeboliceq}
			\iint_{\mathcal{T}_{\ell}}f^2\left\{(\widehat{K}_f)_x\phi_x-(\widehat{K}_f)_t\phi_t\right\}=0\qquad \forall \phi \in C^{\infty}_0(\mathcal{T}_{\ell}).
		\end{equation}
		Therefore, $\widehat{K}_f$ is a weak solution of the hyperbolic equation
		\begin{equation}\label{eq:hyperbolic2}
			\frac{\partial}{\partial x}\left(f^2(x)\frac{\partial \widehat{K}_f(x,t)}{\partial x}\right)=f^2(x)\frac{\partial^2\widehat{K}_f(x,t)}{\partial t^2},\quad (x,t)\in \mathcal{Q}_{\ell},
		\end{equation}
		satisfying the Goursat conditions \eqref{eq:goursatimpedance}.
	\end{remark}

	\section{Conclusions}
	A new factorization of the distributional Schrödinger operator, based on Polya's factorization, has been established. Using this factorization, a Volterra integral transmutation operator was constructed, with its kernel represented as a Fourier–Legendre series. This representation allows the solutions of the Schrödinger equation to be expressed as a Neumann series of spherical Bessel functions, with coefficients that can be computed via a simple recursive integration procedure. Furthermore, Polya factorization enables the derivation of an integro-differential transmutation operator for the Sturm–Liouville equation in impedance form. The connection between the Schrödinger and impedance operators via Polya factorization provides a unified framework for their theories and spectral problems. In practical spectral problem solving, NSBF series have proven to be a highly effective tool for the numerical treatment of both direct and inverse Sturm–Liouville problems, applicable to both the impedance and Schrödinger equations \cite{NSBF1,vkinverse1,vkinverse2}.
    \newline 
    
    {\bf Acknowlegdments:}
	The author thanks to Instituto de Matemáticas de la U.N.A.M. Unidad Querétaro (México), where this
	work was developed, and the SECIHTI for their support through the program {\it Estancias Posdoctorales por México Convocatoria 2023 (I)}.
	
	The author also thanks Prof. Sergii M. Torba for his valuable discussions and insightful suggestions concerning Section 7.
	\newline
	
	{\bf Conflict of interest:} This work does not have any conflict of interest.
    \newline
    
    
	
\end{document}